\documentclass[11pt]{amsart}

\usepackage{amsmath,amssymb,latexsym,amsthm,newlfont,enumerate}
\usepackage{hyperref}
\usepackage[all]{xy}

\theoremstyle{plain}

\newtheorem{thm}{Theorem}[section]

\newtheorem{pro}[thm]{Proposition}

\newtheorem{lem}[thm]{Lemma}

\newtheorem{cor}[thm]{Corollary}

\newtheorem*{GenNonvanCon}{Generalised Nonvanishing Conjecture}
\newtheorem*{WeakGenNonvanCon}{Weak Generalised Nonvanishing Conjecture}

\newtheorem{thmA}{Theorem}

\newtheorem{corA}[thmA]{Corollary}

\theoremstyle{definition}

\newtheorem{dfn}[thm]{Definition}

\newtheorem{rem}[thm]{Remark}

\theoremstyle{remark}

\newcommand{\Z}{\mathbb{Z}}
\newcommand{\N}{\mathbb{N}}

\newcommand{\R}{\mathbb{R}}
\newcommand{\Q}{\mathbb{Q}}

\newcommand{\OO}{\mathcal{O}}

\DeclareMathOperator{\codim}{codim}

\DeclareMathOperator{\Exc}{Exc}
\DeclareMathOperator{\mult}{mult}

\DeclareMathOperator{\Supp}{Supp}

\begin{document}
\title[On Generalised Abundance, II]{On Generalised Abundance, II}

%\date{\today}
\author{Vladimir Lazi\'c}
\address{Fachrichtung Mathematik, Campus, Geb\"aude E2.4, Universit\"at des Saarlandes, 66123 Saarbr\"ucken, Germany}
\email{lazic@math.uni-sb.de}

\author{Thomas Peternell}
\address{Mathematisches Institut, Universit\"at Bayreuth, 95440 Bayreuth, Germany}
\email{thomas.peternell@uni-bayreuth.de}

\thanks{
Lazi\'c was supported by the DFG-Emmy-Noether-Nachwuchsgruppe ``Gute Strukturen in der h\"oherdimensionalen birationalen Geometrie". Peternell was supported by the DFG grant ``Zur Positivit\"at in der komplexen Geometrie". We would like to thank S.\ Schreieder for asking a question about the paper \cite{LP18} at a seminar in Munich in December 2017 which motivated this paper, and to J.-P.\ Demailly for useful discussions.
}

\begin{abstract}
In our previous work, we introduced the Generalised Nonvanishing Conjecture, which generalises several central conjectures in algebraic geometry. In this paper, we derive some surprising nonvanishing results for pluricanonical bundles which were not predicted by the Minimal Model Program, by making progress towards the Generalised Nonvanishing Conjecture in every dimension. The main step is to establish that a somewhat stronger version of the Generalised Nonvanishing Conjecture holds almost always in the presence of metrics with generalised algebraic singularities, assuming the Minimal Model Program in lower dimensions.
\end{abstract}

\maketitle
\setcounter{tocdepth}{1}
\tableofcontents

\section{Introduction}

In this paper we study the Generalised Nonvanishing Conjecture in every dimension, and derive some surprising nonvanishing results for pluricanonical bundles which were not predicted by the Minimal Model Program.

\begin{GenNonvanCon}
Let $(X,\Delta)$ be a klt pair such that $K_X+\Delta$ is pseudoeffective. Let $L$ be a nef $\Q$-divisor on $X$. Then for every $t\geq0$ the numerical class of the divisor $K_X+\Delta+tL$ belongs to the effective cone.
\end{GenNonvanCon}

If the numerical class of a divisor belongs to the effective cone, then we call the divisor \emph{num-effective}, see Section \ref{sec:prelim}.
The Generalised Nonvanishing Conjecture as well as the Generalised Abundance Conjecture were introduced in our paper \cite{LP18a} as generalisations of several central conjectures in algebraic geometry. 

This conjecture contains as special cases the Nonvanishing Conjecture -- often viewed as part of the Abundance Conjecture -- and the nonvanishing part of the Semiampleness Conjecture (for nef line bundle on varieties with numerically trivial canonical class).

\subsection{Numerical versus linear equivalence} 

Constructing many non-tri\-vi\-al effective or basepoint free divisors on a projective variety is one of the main goals of the Minimal Model Program. Assume, for instance, that $(X,\Delta)$ is a projective klt pair such that $K_X+\Delta$ is semiample. It is a basic question to determine which $\Q$-divisors $D$ in the numerical class of $K_X+\Delta$ are effective (that is, $\kappa(X,D)\geq0$) or semiample. If the numerical and $\Q$-linear equivalence on $X$ coincide, i.e.\ if $H^1(X,\OO_X) = 0 $, then \emph{all} such $D$ are trivially effective, respectively semiample, and one would expect that this is the only case where such behaviour happens.

Our first main result is that, surprisingly, such behaviour happens \emph{almost always}. More precisely, whether each such $D$ is effective, respectively semiample, depends not (only) on the vanishing of the cohomology group $H^1(X,\OO_X)$, but on the behaviour of the Euler-Poincar\'e characteristic of the structure sheaf of $X$. 

We obtain the following unconditional results in dimensions at most $3$, which were previously not even conjectured.

\begin{thmA}\label{cor:F}
Let $(X,\Delta)$ be a projective klt pair of dimension at most $3$ such that $K_X+\Delta$ is pseudoeffective. Assume that $\chi(X,\OO_X)\neq0$.
\begin{enumerate}
\item[(i)] Then for every $\Q$-divisor $G$ with $K_X+\Delta\equiv G$ we have $\kappa(X,G)\geq0$. 
\item[(ii)] If $K_X+\Delta$ is semiample, then every $\Q$-divisor $G$ with $K_X+\Delta\equiv G$ is semiample.
\end{enumerate}
\end{thmA}

Theorem \ref{cor:F} is a special case of Corollaries \ref{cor:surfaces} and \ref{cor:threefolds} below.

Another instance where our results hold unconditionally \emph{in every dimension} is the case where the numerical dimension is $1$; this is the first highly nontrivial case beyond those when the numerical dimension is $0$ or maximal. 

\begin{thmA} \label{thm:G} 
Let $(X,\Delta)$ be a projective terminal pair of dimension $n$ such that $K_X + \Delta$ is nef and let $L$ be a nef $\Q$-Cartier divisor on $X$. Assume that $\nu(X,K_X+\Delta+L) = 1$ and $\chi(X,\OO_X) \neq 0$. 

Then $K_X+\Delta+tL$ is num-effective for all $t\geq0$.
\end{thmA}

This is proved in Section \ref{sec:num1}.

As an illustration, consider a terminal Calabi-Yau pair $(X,\Delta)$, i.e.\ a terminal pair such that $K_X + \Delta \equiv 0$. Let $L$ be a nef divisor on $X$ of numerical dimension $1$. If $\chi(X,\OO_X) \neq 0$, then $L$ is num-effective.

\subsection{Metrics with algebraic singularities} 
Theorem \ref{cor:F} above is a special instance of more general results which work in all dimensions but require inductive hypotheses. 

Let $(X,\Delta)$ be a klt pair such that $K_X+\Delta$ is nef and $X$ is not uniruled, and let $L$ be a nef $\Q$-divisor on $X$. Our main technical result, Theorem \ref{thm:nonvanishingForms} below, gives a general criterion for the existence of sections of some multiple of $K_X+\Delta+L$. The criterion generalises the main results of \cite{LP18}: the proof uses the birational stability of the cotangent bundle from \cite{CP15} (this is where the non-uniruledness of $X$ is necessary) together with the very carefully chosen MMP techniques in Theorem \ref{thm:MMPtwist} (this is where the nefness of $K_X+\Delta$ is used). 

We then apply Theorem \ref{thm:nonvanishingForms} to line bundles possessing metrics with \emph{generalised algebraic singularities}; the importance of this class of metrics is explained in Section \ref{sec:prelim}. We expect these methods to be crucial in the resolution of the Generalised Nonvanishing Conjecture. 

In Section \ref{sec:algsing2} we prove our central inductive statement.

\begin{thmA}\label{thm:C}
Assume the termination of flips in dimensions at most $n-1$, and the Abundance Conjecture in dimensions at most $n-1$.

Let $(X,\Delta)$ be a projective klt pair of dimension $n$ such that $K_X+\Delta$ is pseudoeffective, and let $L$ be a nef $\Q$-divisor on $X$. Suppose that $K_X+\Delta$ and $K_X+\Delta+L$ have singular metrics with generalised algebraic singularities and semipositive curvature currents, and that $\chi(X,\OO_X)\neq0$.

\begin{enumerate}
\item[(i)] Then for every $\Q$-divisor $L'$ with $L\equiv L'$ there exists a positive rational number $t_0$ such that
$$\kappa(X,K_X+\Delta+tL')\geq0\quad\text{for every }0\leq t\leq t_0.$$
\item[(ii)] Assume additionally the semiampleness part of the Abundance Conjecture in dimension $n$. Then 
$$K_X+\Delta+tL\quad\text{is num-effective for every }t\geq0.$$ 
\end{enumerate}
\end{thmA}

With notation from Theorem \ref{thm:C}, we note that by \cite[Theorem A]{LP18a}, all divisors of the form $K_X+\Delta+L$ should possess singular metrics with generalised algebraic singularities and semipositive curvature currents.

\medskip

We mention a special case, where the conclusion of Theorem \ref{thm:C} is much stronger: when $K_X+\Delta$ and $K_X+\Delta+L$ additionally have  smooth metrics with semipositive curvature, or, more generally, singular metrics with semipositive curvature 
currents and vanishing Lelong numbers, then $K_X+\Delta+L$ is actually semiample. This is Theorem \ref{thm:semipositive}, and the proof uses crucially the main result of \cite{GM17}.

As a corollary of Theorem \ref{thm:C}, we obtain the following:

\begin{corA}\label{cor:D}
Let $(X,\Delta)$ be a projective klt pair of dimension $4$ such that $K_X+\Delta$ is pseudoeffective and let $L$ be a nef $\Q$-divisor on $X$. Suppose that $K_X+\Delta$ and $K_X+\Delta+L$ have singular metrics with generalised algebraic singularities and semipositive curvature currents, and that $\chi(X,\OO_X)\neq0$.

Then there exists a positive rational number $t_0$ such that 
$$\kappa(X,K_X+\Delta+tL)\geq0\quad\text{for every }0\leq t\leq t_0.$$
\end{corA}

Theorem \ref{thm:C} has an unexpected consequence for the Minimal Model Program in terms of the existence of sections of divisors numerically equivalent to adjoint divisors, generalising Theorem \ref{cor:F} to every dimension. The following result is a special case of Corollary \ref{cor:numequiv} and Theorem \ref{thm:semipositive1}.

\begin{thmA}\label{thm:E}
Assume the termination of flips in dimensions at most $n-1$, and the existence of good models in dimensions at most $n$.

Let $(X,\Delta)$ be a projective klt pair of dimension $n$ such that $K_X+\Delta$ is pseudoeffective. Assume that $\chi(X,\OO_X)\neq0$.
\begin{enumerate}
\item[(i)] Then for every $\Q$-divisor $G$ with $K_X+\Delta\equiv G$ we have $\kappa(X,G)\geq0$. 
\item[(ii)] If $K_X+\Delta$ is semiample, then every $\Q$-divisor $G$ with $K_X+\Delta\equiv G$ is semiample.
\end{enumerate}
\end{thmA}

\subsection{The Weak Generalised Nonvanishing Conjecture}

As an important part of this paper on which the previous results depend, we show that the assumptions of the Generalised Nonvanishing Conjecture can be weakened considerably.  In fact, we reduce the Generalised Nonvanishing Conjecture to the following.

\begin{WeakGenNonvanCon} 
Let $(X,\Delta)$ be a $\Q$-fac\-to\-ri\-al projective klt pair of dimension $n$ such that $X$ is not uniruled and $K_X + \Delta $ is nef. Let $L$ be a nef $\Q$-divisor on $X$ such that $n(X,K_X + \Delta + L) = n$. Then for every $t\geq0$ the numerical class of the divisor $K_X+\Delta+tL$ belongs to the effective cone.
\end{WeakGenNonvanCon}

Here, for any $\Q$-Cartier divisor $D$ on $X$, we denote by $n(X,D)$ the nef dimension of $D$, see Section \ref{sec:prelim}. The condition $n(X,D) = \dim X $ is equivalent to saying that $X$ cannot be covered by curves $C$ such that $D \cdot C = 0$. Conjecturally, when $n(X,K_X + \Delta + L) = \dim X$ as above, then $K_X+\Delta+L$ is big. 

This conjecture is weaker than the Generalised Nonvanishing Conjecture in three different ways: the variety $X$ is not uniruled; the pair $(X,\Delta)$ is a minimal model; and the nef dimension of $K_X+\Delta+L$ is maximal; see Section \ref{sec:prelim} for details.

\medskip

In Section \ref{sec:reductions} we prove:

\begin{thmA}\label{thm:mainreduction}
Assume the termination of klt flips in dimensions at most $n-1$, the Abundance Conjecture for klt pairs in dimensions at most $n-1$, the existence of minimal models of klt pairs in dimension $n$, and the Weak Generalised Nonvanishing Conjecture in dimensions at most $n$.

Let $(X,\Delta)$ be an $n$-dimensional klt pair such that $K_X+\Delta$ is pseudoeffective and let $L$ be a nef $\Q$-divisor on $X$. 
\begin{enumerate}
\item[(i)] Then there exists a positive rational number $t_0$ such that $K_X+\Delta+tL$ is num-effective for every $0\leq t\leq t_0$.
\item[(ii)] Assume additionally the semiampleness part of the Abundance Conjecture in dimension $n$. Then $K_X+\Delta+tL$ is num-effective for every $t\geq0$, that is, the Generalised Nonvanishing Conjecture holds in dimension $n$. 
\end{enumerate}
\end{thmA}

The heart of the argument is in Theorem \ref{thm:Support}, which essentially deals with the situation when $X$ is not uniruled, together with Theorem \ref{thm:B}, which deals with the situation when the nef dimension $n(X,K_X+\Delta+L)$ is not maximal.
 
\medskip

A more general version of the Generalised Nonvanishing Conjecture was proposed by Han and Liu \cite{HL18}. They do not require the divisor $K_X+\Delta$ to be pseudoeffective and allow log canonical singularities, and they confirm the conjecture in dimension $2$.

\section{Preliminaries}\label{sec:prelim}

In this section we collect notation and technical lemmas that we use later. 

\subsection{Num-effectivity, models, MMP}
Unless otherwise stated, a $\Q$-di\-vi\-sor on a projective variety is assumed to be a Weil $\Q$-divisor. We write $D \geq 0$ for an effective $\Q$-divisor $D$ on a normal variety $X$. 

A $\Q$-divisor $L$ on a projective variety $X$ is \emph{num-effective} if the numerical class of $L$ belongs to the effective cone of $X$. A $\Q$-divisor $L$ on a projective variety $X$ is \emph{num-semiample} if there exists a semiample $\Q$-divisor $L'$ on $X$ such that $L\equiv L'$.

We need the following easy lemma.

\begin{lem}\label{lem:easylemma}
Let $D$ and $D'$ be effective $\Q$-Cartier $\Q$-divisors on a normal projective variety $X$ and let $L$ be a $\Q$-Cartier $\Q$-divisor on $X$. Assume that $\Supp D=\Supp D'$. Then:
\begin{enumerate}
\item[(a)] $D+tL$ is big for all $t>0$ if and only if $D'+tL$ is big for all $t>0$,
\item[(b)] there exists $t_0>0$ such that $D+tL$ is num-effective for all $0\leq t\leq t_0$ if and only if there exists $t_0'>0$ such that $D'+tL$ is num-effective for all $0\leq t\leq t_0'$.
\end{enumerate}
\end{lem}

\begin{proof}
Assume that $D+tL$ is big for all $t>0$. Pick a positive integer $m$ such that $mD'\geq D$. Then for each $t>0$, the divisor
$$m(D'+tL)=(D+mtL)+(mD'-D)$$
is big. This shows (a).

Now assume that $D+tL$ is num-effective for all $0\leq t\leq t_0$, for some positive $t_0$. Pick a positive integer $m$ such that $mD'\geq D$. Then for every $0\leq t\leq t_0/m$ there exists an effective divisor $F$ such that $D+mt_0 L\equiv F$, hence 
$$m(D'+tL)=D+mtL+(mD'-D)\equiv F+(mD'-D)\geq0,$$
which gives (b).
\end{proof}

A \emph{fibration} is a projective surjective morphism with connected fibres between two normal varieties.

If $f\colon X\to Y$ is a surjective morphism of normal projective varieties and if $D$ is a $\Q$-divisor on $X$, then $D$ is \emph{$f$-exceptional} if $\codim_Y f(\Supp D) \geq 2$.

\begin{dfn}
Let $X$ and $Y$ be normal projective varieties, and let $D$ be a $\Q$-Cartier $\Q$-divisor on $X$. Let $\varphi\colon X\dashrightarrow Y$ be a birational contraction, and assume that $\varphi_*D$ is $\Q$-Cartier. Then $\varphi$ is \emph{$D$-non-positive} (respectively \emph{$D$-negative}) if there exists a smooth resolution of indeterminacies $(p,q)\colon W\to X\times Y$ of $\varphi$ such that
$$ p^*D\sim_\Q q^*\varphi_*D+E,$$
where $E\geq0$ is a $q$-exceptional $\Q$-divisor (respectively, $E\geq0$ is a $q$-exceptional $\Q$-divisor and $\Supp E$ contains the proper transform of every $\varphi$-exceptional divisor). If additionally $\varphi_*D$ is semiample, the map $\varphi$ is a \emph{good model} for $D$.
\end{dfn}

\begin{lem}\label{lem:MMPnumeff}
Let $\varphi\colon X\dashrightarrow Y$ be a birational contraction between normal projective varieties, where $X$ has rational singularities. Let $D$ be a $\Q$-Cartier $\Q$-divisor on $X$ such that $\varphi_*D$ is $\Q$-Cartier on $Y$ and such that the map $\varphi$ is $D$-non-positive. If $\varphi_*D$ is num-effective, then $D$ is num-effective.
\end{lem}

\begin{proof}
Let $(p,q)\colon W\to X\times Y$ be a smooth resolution of indeterminacies of $\varphi$. Since $\varphi$ is $D$-non-positive, there exists an effective $q$-exceptional divisor $E$ on $W$ such that
$$p^*D\sim_\Q q^*\varphi_*D+E.$$
If there exists an effective $\Q$-divisor $G$ on $Y$ such that $\varphi_*D\equiv G$, then
$$p^*D\equiv q^*G+E\geq0,$$
hence $D$ is num-effective by \cite[Lemma 2.14]{LP18a}.
\end{proof}

A \emph{pair} $(X,\Delta)$ consists of a normal variety $X$ and a $\Q$-divisor $\Delta\geq0$ such that $K_X+\Delta$ is $\Q$-Cartier. The standard reference for the foundational definitions and results on the singularities of pairs and the Minimal Model Program is \cite{KM98}, and we use these freely in this paper. 

We recall additionally that flips for klt pairs exist by \cite[Corollary 1.4.1]{BCHM}. We also use throughout the paper that for every projective klt pair $(X,\Delta)$, a small $\Q$-factorialisation of $(X,\Delta)$ exists, see \cite[Corollary 1.37]{Kol13}; and that a terminalisation of $(X,\Delta)$ exists, see \cite[paragraph after Corollary 1.4.3]{BCHM}. A small $\Q$-factorialisation of $X$ is an isomorphism over the $\Q$-factorial locus of $X$, see \cite[1.40]{Deb01}.

The following result is well-known.

\begin{lem}\label{lem:numlintrivial}
Let $(X,\Delta)$ be a projective klt pair and let $f\colon X\to Y$ be a fibration to a projective variety $Y$. If $K_X+\Delta\equiv_f0$, then $K_X+\Delta\sim_{f,\Q}0$.
\end{lem}

\begin{proof}
Let $F$ be a general fibre of $f$. Then $(K_X+\Delta)|_F\sim_\Q0$ by \cite[Corollary V.4.9]{Nak04}, and thus the pair $(X,\Delta)$ has a good model over $Y$ by \cite[Theorem 2.12]{HX13}. Therefore, $K_X+\Delta$ is $f$-semiample, and there exists a fibration $\pi\colon X\to Z$ over $Y$ and a $\Q$-divisor $A$ on $Z$ which is ample over $Y$ such that $K_X+\Delta\sim_\Q\pi^*A$. As $K_X+\Delta\equiv_f0$, the map $\pi$ must be an isomorphism, and the lemma follows.
\end{proof}

\subsection{Numerical dimension}\label{subsec:numdim}

We recall the definition of the numerical dimension from \cite{Nak04,Kaw85}.

\begin{dfn}\label{dfn:kappa}
Let $X$ be a normal projective variety, let $D$ be a pseudoeffective $\Q$-Cartier $\Q$-divisor on $X$ and let $A$ be any ample divisor on $X$. Then the {\em numerical dimension\/} of $D$ is
$$\nu(X,D)=\sup\big\{k\in\N\mid \limsup_{m\rightarrow\infty}h^0\big(X, \mathcal O_X(\lfloor mD\rfloor+A)\big)/m^k >0\big\}.$$
When the divisor $D$ is nef, then equivalently
$$\nu(X,D)=\sup\{k\in\N\mid D^k\not\equiv0\}.$$
If $D$ is not pseudoeffective, we set $\nu(X,D)=-\infty$.
\end{dfn}

The Kodaira dimension and the numerical dimension of a $\Q$-divisor $D$ are preserved under any $D$-non-positive birational map, see for instance \cite[\S2.2]{LP18a}. If $X$ is a normal projective variety, and if $D$ and $D'$ are pseudoeffective $\Q$-divisors on $X$ such that $D'-D$ is pseudoeffective, then $\nu(X,D')\geq\nu(X,D)$ by the proof of \cite[Proposition V.2.7]{Nak04}.

If $D$ is a pseudoeffective $\R$-Cartier $\R$-divisor on a smooth projective variety $X$, and if $\Gamma$ is a prime divisor on $X$, we denote by $\sigma_\Gamma(D)$ Nakayama's multiplicity of $D$ along $\Gamma$, see\ \cite[Chapter III]{Nak04} for the definition and the basic properties.

For the proof of the following lemma we refer to \cite[Lemma 2.3]{LP18a}. 

\begin{lem}\label{lem:1}
Let $X$, $X'$, $Y$ and $Y'$ be normal varieties, and assume that we have a commutative diagram
\[
\xymatrix{ 
X' \ar[d]_{\pi'} \ar[r]^{f'} & Y' \ar[d]^{\pi}\\
X \ar[r]_{f} & Y,
}
\]
where $\pi$ and $\pi'$ are projective birational. Let $D$ be a $\Q$-Cartier $\Q$-divisor on $X$ and let $E\geq0$ be a $\pi'$-exceptional $\Q$-Cartier $\Q$-divisor on $X'$. If $F$ and $F'$ are general fibres of $f$ and $f'$, respectively, such that $F'=\pi'^{-1}(F)$, then 
$$\nu(F,D|_F)=\nu\big(F',(\pi'^*D+E)|_{F'}\big).$$
\end{lem}

\subsection{Nef reduction}

In this paper we use crucially the nef reduction map. The following is the main result of \cite{BCE+}.

\begin{thm}\label{thm:nefreduction}
Let $L$ be a nef divisor on a normal projective variety $X$. Then there exists an almost holomorphic dominant rational map $f\colon X\dashrightarrow Y$ with connected fibres to a normal projective variety $Y$, called the \emph{nef reduction of $L$}, such that:
\begin{enumerate}
\item[(i)] $L$ is numerically trivial on all compact fibres $F$ of $f$ with $\dim F=\dim X-\dim Y$,
\item[(ii)] for every very general point $x\in X$ and every irreducible curve $C$ on $X$ passing through $x$ and not contracted by $f$, we have $L\cdot C>0$.
\end{enumerate}
The map $f$ is unique up to birational equivalence of $Y$ and the \emph{nef dimension} $n(X,L)$ of $L$ is defined as the dimension of $Y$:
$$ n(X,L) = \dim Y. $$
\end{thm}

It is immediate from the definition that, with notation above, we have $n(X,L)=\dim X$ if and only if $L\cdot C>0$ for every irreducible curve $C$ on $X$ passing through a very general point on $X$. 

In the following lemmas we study how the nef dimension behaves under morphisms and birational contractions.

\begin{lem}\label{lem:nefredsum}
Let $X$ be a normal projective variety of dimension $n$. Let $F$ and $G$ be pseudoeffective divisors on $X$ such that $(F+G)\cdot C_X>0$ for every curve $C_X$ on $X$ passing through a very general point of $X$. Then for every positive real number $t$ and every curve $C$ on $X$ passing through a very general point of $X$ we have $(F+tG)\cdot C>0$.
\end{lem}

\begin{proof}
The proof is the same as that of \cite[Lemma 2.11]{LP18a}.
\end{proof}

\begin{lem}\label{lem:generalintersections}
Let $f\colon X\to Y$ be a birational morphism between normal projective varieties of dimension $n$. Let $D$ be a nef $\Q$-divisor on $Y$ and let $G=f^*D+E$, where $E\geq0$ is $f$-exceptional. If $G\cdot C_X>0$ for every curve $C_X$ on $X$ passing through a very general point on $X$, then $n(Y,D)=n$.
\end{lem}

\begin{proof}
Denote $V:=f(\Exc f)$. Assume that $G\cdot C_X>0$ for every curve $C_X$ on $X$ passing through a point in a very general subset $U_X\subseteq X$, and assume that $n(Y,D)<n$. Let $\varphi\colon Y\dashrightarrow Z$ be the nef reduction of $D$. Then $\dim Y<n$ and recall that $\varphi$ is almost holomorphic. The set $U_Y=Y\setminus f(X\setminus U_X)$ is a very general subset of $Y$ such that $f^{-1}(U_Y)\subseteq U_X$. If $F$ is a very general fibre of $\varphi$, then $U_Y\cap F$ is a very general subset of $F$, and $\codim_F(F\cap V)\geq2$. Therefore, a very general complete intersection curve $C_Y$ in $F$ avoids the set $F\cap V$ and intersects $U_Y\cap F$. Then for $C:=f^{-1}(C_Y)$ we have $G\cdot C=0$ and $C\cap U_X\neq\emptyset$, a contradiction.
\end{proof}

\begin{lem}\label{lem:nefdim1}
Let $f\colon X\to Y$ be a birational morphism between normal projective varieties of dimension $n$. Let $D$ be a $\Q$-divisor on $Y$ and let $G=f^*D+E$, where $E$ is a pseudoeffective divisor on $X$. Assume that $D\cdot C_Y>0$ for every curve $C_Y$ on $Y$ passing through a very general point on $Y$. Then $G\cdot C_X>0$ for every curve $C_X$ on $X$ passing through a very general point on $X$.
\end{lem}

\begin{proof}
Assume that $D\cdot C_Y>0$ for every curve $C_Y$ on $Y$ passing through a point in a very general subset $U_Y\subseteq Y$. We may assume that $U_Y$ is contained in the locus on $Y$ over which $f$ is an isomorphism. Since $E$ is pseudoeffective, there exists a very general subset $U_X\subseteq X$ such that $E\cdot C_X\geq0$ for every curve $C_X$ on $X$ passing through a point in $U_X$. Then for every curve $C$ on $X$ passing through a point in $U_X\cap f^{-1}(U_Y)$ we have $G\cdot C\geq f^*D\cdot C=D\cdot f(C)>0$. 
\end{proof}

\begin{cor}\label{cor:nefdim}
Let $\varphi\colon X\dashrightarrow Y$ be a birational contraction between normal projective varieties of dimension $n$. Let $D$ be a $\Q$-Cartier $\Q$-divisor on $X$ such that $\varphi_*D$ is $\Q$-Cartier and nef, and assume that $\varphi$ is $D$-non-positive. If $D\cdot C>0$ for every curve $C$ on $X$ passing through a very general point on $X$, then $n(Y,\varphi_*D)=n$.
\end{cor}

\begin{proof}
Let $(p,q)\colon W\to X\times Y$ be a resolution of indeterminacies of $f$. Since $\varphi$ is $D$-non-positive, there exists an effective $q$-exceptional divisor $E$ on $W$ such that
$$p^*D\sim_\Q q^*\varphi_*D+E.$$
By Lemma \ref{lem:nefdim1} we have $p^*D\cdot C_W>0$ for every curve $C_W$ on $W$ passing through a very general point on $W$. But then the conclusion follows from Lemma \ref{lem:generalintersections}.
\end{proof}

\subsection{Metrics on $\Q$-divisors}

We briefly recall the properties of metrics on line bundles on projective varieties. For a more thorough treatment, see \cite{DPS01,De12,Bou04}. In the sequel, if $F$ is a prime divisor on a smooth projective variety $X$, then $[F]$ denotes the current of integration along $F$. If $h$ is a singular metric on a some line bundle, and if $\varphi$ is a corresponding local plurisubharmonic function, then $\mathcal I(h)=\mathcal I(\varphi)$ denote the multiplier ideal associated to $h$, respectively $\varphi$. 

Let $X$ be a smooth projective variety and let $D$ be a Cartier divisor on $X$. Then $D$ is pseudoeffective if and only if $\OO_X(D)$ has a singular metric $h$ whose associated curvature current $\Theta_h(D)$ is semipositive (as a current); we also write $\Theta_h(D)\geq0$.\footnote{Often in the literature such currents are called positive.} Equivalently, $D$ is pseudoeffective if and only if its Chern class $c_1(D)$ contains a closed semipositive current. 

Recall that by Demailly's regularisation technique \cite[\S15.B]{De12}, every pseudoeffective line bundle on a smooth projective variety can be equipped with a sequence of singular metrics with algebraic singularities whose curvature currents converge to a semipositive current; for the definition of metrics with algebraic singularities, see for instance \cite{LP18}. Thus, line bundles possessing singular metrics with algebraic singularities and semipositive curvature current form an important subclass of pseudoeffective line bundles; for instance, all hermitian semipositive line bundles belong to this class, as well as all num-effective line bundles. 

In this paper, we consider a more general class of metrics with generalised algebraic singularities. This class of singularities behaves very well under the operations of the Minimal Model Program.

\begin{dfn} \label{dfn:metrics}
Let $X$ be a normal complex projective variety and $D$ a Cartier divisor on $X$. We say that $D$ has a metric with \emph{generalised algebraic singularities} and semipositive curvature current, if there exists a resolution of singularities $\pi\colon Y \to X$ such that the line bundle $\pi^*\OO_X(D)$ has a singular metric $h$ whose curvature current is semipositive, and the Siu decomposition (see, for instance, \cite[\S2.2]{Bou04}) of $\Theta_h(\pi^*D)$ has the form
$$\Theta_h(\pi^*D)=\Theta+\sum\lambda_j [D_j],$$
where $\Theta\geq0$ is a current whose all Lelong numbers are zero, and $\sum\lambda_j [D_j]$ is a $\Q$-divisor on $Y$. We then say that the metric $h$ \emph{descends} to $Y$.

If $D$ is a $\Q$-Cartier divisor, we say that $D$ has a metric with generalised algebraic singularities and semipositive curvature current, if there exists a positive integer $m$ such that $mD$ is Cartier and $mD$ has a metric with generalised algebraic singularities and semipositive curvature current. 
\end{dfn} 

Since the pullback of a $(1,1)$-current whose all Lelong numbers are zero by a proper morphism is again a $(1,1)$-current whose all Lelong numbers are zero by \cite[Corollary 4]{Fav99}, the pullback of any $(1,1)$-current with generalised algebraic singularities is again a $(1,1)$-current with generalised algebraic singularities.

The class of metrics (or currents) with generalised algebraic singularities contains several other important classes of metrics: metrics with algebraic singularities, metrics whose all Lelong numbers are zero (in particular, all hermitian semipositive metrics are in this class), as well as currents of integration along effective divisors. 

\begin{rem}\label{rem:WZD}
In the notation from Definition \ref{dfn:metrics}, the current $\Theta$ is nef by \cite[Corollary 6.4]{Dem92}. Therefore, $D$ has a weak Zariski decomposition in the sense of \cite[Definition 1.3]{Bir12b}.
\end{rem}

The following lemma allows to calculate the multiplier ideal of a metric with generalised algebraic singularities.

\begin{lem}
Let $X$ be a complex manifold and let $\varphi$ and $\psi$ be two plurisubharmonic functions on $X$. Let $x$ be a point in $X$ and assume that the Lelong number of $\varphi$ at $x$ is zero. Then
$$\mathcal I(\varphi+\psi)_x=\mathcal I(\psi)_x.$$
In particular, let $\mathcal L$ be a line bundle on a projective manifold $X$ and let $h$ be a singular metric on $\mathcal L$ with semipositive curvature current whose Siu decomposition is $$\Theta_h(\mathcal L)=\Theta+\sum\lambda_i[D_i],$$ 
where all Lelong numbers of $\Theta$ are zero and $\sum D_i$ is a simple normal crossings divisor. Then
$$\textstyle\mathcal I(h^{\otimes m})=\OO_Y\big({-}\sum\lfloor m\lambda_j\rfloor D_j\big)\quad\text{for every positive integer }m.$$
\end{lem}

\begin{proof}
Since $\varphi $ is locally bounded from above, we have 
$$\mathcal I(\varphi+\psi)_x\subseteq\mathcal I(\psi)_x,$$
and it suffices to show the reverse inclusion; note that by \cite[Theorem 2.6(ii)]{DEL00} one has even the stronger inclusion
$$ \mathcal I(\varphi+\psi)_x\subseteq \mathcal I(\varphi)_x \cdot \mathcal I(\psi)_x.$$

By \cite{GZ15} there exists a positive rational number $p>1$ such that $\mathcal I(p\psi)=\mathcal I(\psi)$, and let $q>1$ be a rational number such that $\frac1p+\frac1q=1$. Fix $f\in\mathcal I(p\psi)_x$, and we may assume that $f\in\mathcal I(p\psi)(U)$ for an open set $x\in U\subseteq X$. Since $q\varphi$ has the Lelong number zero at $x$, we have $\mathcal I(q\varphi)_x=\OO_{X,x}$ by Skoda's lemma \cite[Lemma 5.6]{De12}, hence we may assume that $|f|^2e^{-2q\varphi}$ is integrable on $U$. Therefore, H\"older's inequality gives
$$\int_U|f|^2e^{-2(\varphi+\psi)}\leq\left(\int_U|f|^2e^{-2q\varphi}\right)^\frac{1}{q}\left(\int_U|f|^2e^{-2p\psi}\right)^\frac{1}{p}<\infty,$$
hence $f\in\mathcal I(\varphi+\psi)(U)$, as desired.
\end{proof}

The following results connect metrics with generalised algebraic singularities with birational geometry. 

\begin{lem}\label{lem:Siu}
Let $f\colon X\to Y$ be a surjective morphism from a smooth projective variety $X$ to a normal projective variety $Y$. Let $D$ be a pseudoeffective $\Q$-Cartier divisor on $Y$ and let $G=f^*D+\sum e_iE_i$, where $E_i$ are $f$-exceptional prime divisors on $X$ and $e_i\geq0$ are rational numbers. Let $T$ be a closed semipositive current in $c_1(G)$. 

Then $T-\sum e_i[E_i]$ is a semipositive current. 
\end{lem}

\begin{proof}
Let $T_{\min}$ be a closed semipositive current with minimal singularities in $c_1(G)$, see for instance \cite[\S2.8]{Bou04}. For each $i$, let $\lambda_{\min,i}$ and $\lambda_i$ denote the Lelong numbers of $T_{\min}$, respectively $T$, along $E_i$. Then we have 
$$\sigma_{E_i}(G)\leq \lambda_{\min,i}$$
by \cite[Proposition 3.6(i)]{Bou04}, and $\lambda_{\min,i}\leq \lambda_i$ by the definition of the current with minimal singularities. On the other hand, we have 
$$e_i\leq \sigma_{E_i}(G)$$
by \cite[Lemma 2.16]{GL13}. The lemma follows by combining all this with the Siu decomposition of $T$.
\end{proof}

\begin{cor}\label{cor:algsinMMP}
Let $X$ be a normal projective variety and let $D$ be a $\Q$-Cartier $\Q$-divisor on $X$. Let $f\colon X\dashrightarrow Y$ be a birational contraction such that $f_*D$ is $\Q$-Cartier, and assume that $f$ is $D$-non-positive. If $D$ has a singular metric with generalised algebraic singularities and semipositive curvature current, then $f_*D$ has a singular metric with generalised algebraic singularities and semipositive curvature current.
\end{cor}

\begin{proof}
Let $(p,q)\colon W\to X\times Y$ be a resolution of indeterminacies of $f$ such that $W$ is smooth. Then there exist $q$-exceptional prime divisors $E_i$ on $W$ and rational numbers $e_i\geq0$ such that
\begin{equation}\label{eq:compare}
p^*D\sim_\Q q^*f_*D+\sum e_iE_i.
\end{equation}
We may assume that the given metric with generalised algebraic singularities on $D$ descends to $W$. Then the associated curvature current $T$ has the Siu decomposition of the form
$$T=\Theta+\sum\lambda_i [D_i],$$
where $\Theta\geq0$ has all Lelong numbers zero, $\sum D_i$ is a simple normal crossings divisor and $\lambda_i$ are positive rational numbers for all $i$. By \eqref{eq:compare} and by Lemma \ref{lem:Siu} we have $\sum\lambda_i D_i\geq \sum e_iE_i$, thus
$$\Theta+\sum \lambda_i [D_i]-\sum e_i[E_i]$$
is a semipositive curvature current in $c_1(q^*f_*D)$. This proves the result.
\end{proof}

We use below crucially the following hard Lefschetz theorem from \cite[Theorem 0.1]{DPS01}. 

\begin{thm}\label{thm:DPS}
Let $X$ be a compact K\"ahler manifold of dimension $n$ with a K\"ahler form $\omega$. Let $\mathcal L$ be a pseudoeffective line bundle on $X$ with a singular hermitian metric $h$ such that $\Theta_h(\mathcal L) \geq 0$. Then for every nonnegative integer $q$ the morphism
\[
\xymatrix{ 
H^0\big(X,\Omega^{n-q}_X\otimes\mathcal  L\otimes\mathcal I(h)\big) \ar[r]^{\omega^q\wedge\bullet} & H^q\big(X, \Omega^n_X\otimes \mathcal L\otimes\mathcal I(h)\big)
}
\]
is surjective.
\end{thm}

\section{Reduction to the Weak~Generalised~Nonvanishing~Conjecture}\label{sec:reductions}

In this section we prove Theorem \ref{thm:mainreduction}. In other words, in order to prove the Generalised Nonvanishing Conjecture, we may assume that  $X$ is not uniruled, that $K_X +\Delta$ is nef, and that the nef dimension of $K_X+\Delta+L$ is maximal. 

We achieve this in several steps. An important step is the proof of the following result announced in the introduction.

\begin{thm} \label{thm:B}
Assume the termination of flips in dimensions at most $n-1$, the Abundance Conjecture in dimensions at most $n-1$, and the Generalised Nonvanishing Conjecture in dimensions at most $n-1$.

Let $(X,\Delta)$ be a projective klt pair of dimension $n$ with $K_X+\Delta$ pseudoeffective, and let $L$ be a nef Cartier divisor on $X$ such that $K_X+\Delta+L$ is nef and $n(X,K_X+\Delta+L)<n$.

Then $K_X+\Delta+tL$ is num-effective for every $t\geq0$.
\end{thm}

The proof of Theorem \ref{thm:B} is similar to the proof of \cite[Theorem 5.3]{LP18a}, and the proof of Theorem \ref{thm:Support} follows the strategy of \cite{DL15}.

\medskip

We start with the following lemma.

\begin{lem}\label{lem:termination}
Assume the Weak Generalised Nonvanishing Conjecture in dimension $n$, the existence of minimal models of klt pairs in dimension $n$, the Abundance Conjecture in dimensions at most $n-1$, and the termination of flips in dimension $n-1$. 

Assume Theorem \ref{thm:B} in dimension $n$. Let $(X,\Delta)$ be a klt pair of dimension $n$ such that $K_X+\Delta$ is pseudoeffective. Then $\kappa(X,K_X+\Delta)\geq0$ and every sequence of flips of $(X,\Delta)$ terminates.
\end{lem}

\begin{proof}
By \cite[Corollary 1.2]{HMX14} it suffices to show that $\kappa(X,K_X+\Delta)\geq0$. By our assumptions, there exists a minimal model of $(X,\Delta)$. Hence, replacing $(X,\Delta)$ by this minimal model, we may assume that $K_X+\Delta$ is nef. 

By Theorem \ref{thm:B} and by \cite[Theorem 0.1]{CKP12} we have $\kappa(X,K_X+\Delta)\geq0$ if $n(X,K_X+\Delta)<n$. Therefore, we may assume that $n(X,K_X+\Delta)=n$. If $X$ is not uniruled, then we conclude by the Weak Nonvanishing Conjecture. If $X$ is uniruled, by passing to a terminalisation of $(X,\Delta)$ we may assume that $(X,\Delta)$ is terminal. Then $K_X$ is not pseudoeffective, and we conclude that $\kappa(X,K_X+\Delta)\geq0$ by \cite[Theorem 3.3]{DL15}.
\end{proof}

\subsection{Achieving nefness of $K_X+\Delta+L$}\label{subsec:nefness} 

The first step is to achieve the nefness of $K_X+\Delta+L$. 

\begin{pro}\label{pro:contocon}
Let $(X,\Delta)$ be a projective klt pair of dimension $n$ such that $K_X+\Delta$ is pseudoeffective. Assume that any sequence of flips of the pair $(X,\Delta)$ terminates. Let $L$ be a nef Cartier divisor on $X$ and let $m>2n$ be a positive integer. Then there exists an $L$-trivial $(K_X+\Delta)$-MMP $\varphi\colon X\dashrightarrow Y$ such that $K_Y+\varphi_*\Delta+m\varphi_*L$ is nef.
\end{pro}

\begin{proof}
Apart from the different assumptions, this is \cite[Proposition 4.1]{LP18a}. The only thing to note is that the proof of that result uses the termination of flips only for the given pair $(X,\Delta)$, and not for all klt pairs.
\end{proof}

Therefore, we obtain:

\begin{pro} \label{pro:reduction}
Assume the Weak Nonvanishing Conjecture in dimension $n$, the existence of minimal models of klt pairs in dimension $n$, the Abundance Conjecture (for klt pairs) in dimensions at most $n-1$, and the termination of flips in dimension $n-1$. 

Assume Theorem \ref{thm:B} in dimension $n$. In order to prove Theorem \ref{thm:mainreduction}, it suffices to assume that $L$ is Cartier and $K_X + \Delta + L$ is nef. 
\end{pro}

\begin{proof} 
First note that any sequence of flips of the pair $(X,\Delta)$ terminates by Lemma \ref{lem:termination}. Then by Proposition \ref{pro:contocon} there exists an $L$-trivial $(K_X+\Delta)$-MMP $\varphi\colon X\dashrightarrow Y$ such that $K_Y+\varphi_*\Delta+m\varphi_*L$ is nef for some fixed $m\gg0$. Denote $\Delta_Y:=\varphi_*\Delta$ and $L_Y:=m\varphi_*L$. Then $K_Y+\Delta_Y+L_Y$ is nef and by Lemma \ref{lem:MMPnumeff} it suffices to show that $K_Y+\Delta_Y+tL_Y$ is num-effective for every $t\geq0$.
\end{proof} 

\subsection{Non-maximal nef dimension}\label{subsec:nonmaxnefdim}

We first deal with the case of non-maximal nef dimension -- we prove Theorem \ref{thm:B}. More precisely, we have the following inductive step.

\begin{thm}\label{thm:induction}
Assume the termination of flips in dimensions at most $n-1$, the Abundance Conjecture in dimensions at most $n-1$, and the Weak Generalised Nonvanishing Conjecture in dimensions at most $n-1$.

Then Theorem \ref{thm:mainreduction} in dimension $n-1$ implies Theorem \ref{thm:B} in dimension $n$.
\end{thm}

\begin{proof}
Let $(X,\Delta)$ be a projective klt pair of dimension $n$ with $K_X+\Delta$ pseudoeffective, and let $L$ be a nef Cartier divisor on $X$ such that $K_X+\Delta+L$ is nef and $n(X,K_X+\Delta+L)<n$. We need show that $K_X+\Delta+tL$ is num-effective for every $t\geq0$.

\medskip

\emph{Step 1.}
In this step we show that we may assume the following:

\medskip

\emph{Assumption 1.} There exist a morphism $\varphi\colon X\to Y$ to a smooth projective variety $Y$ with $\dim Y<\dim X$ and a nef $\Q$-divisor $L_Y$ on $Y$ such that $L\equiv \varphi^*L_Y$ and $\nu\big(X,(K_X+\Delta)|_F\big)=0$ for a general fibre $F$ of $\varphi$. {\it However, we may not any more assume that $K_X+\Delta+L$ is nef.}

\medskip

To this end, first note that by passing to a small $\Q$-factorialisation and by \cite[Lemma 2.14]{LP18a}, we may assume first that $X$ is $\Q$-factorial. Fix a positive integer $m>2n$. Then 
$$n(X,K_X+\Delta+mL)<n$$ 
by Lemma \ref{lem:nefredsum}. Let 
$$\varphi\colon X\dashrightarrow Y$$
be the nef reduction of $K_X+\Delta+mL$. Then $\dim Y = n(X,K_X+\Delta+mL) <n$, and $K_X+\Delta+mL$ is numerically trivial on a general fibre of $\varphi$. By \cite[Lemma 2.8]{LP18a} this implies that both $K_X+\Delta$ and $L$ are numerically trivial on a general fibre of $\varphi$. 

Recall that $\varphi$ is almost holomorphic. Let $(\widehat\pi,\widehat\varphi)\colon\widehat{X}\to X\times Y$ be a smooth resolution of indeterminacies of the map $\varphi$. Write 
$$K_{\widehat{X}}+\widehat{\Delta}\sim_\Q\widehat{\pi}^*(K_X+\Delta)+\widehat{E},$$
where $\widehat{\Delta}$ and $\widehat{E}$ are effective $\Q$-divisors without common components. Then if $\widehat F$ is a general fibre of $\widehat \varphi$, then $\nu\big(\widehat F,(K_{\widehat{X}}+\widehat{\Delta})|_{\widehat{F}}\big)=0$ by Lemma \ref{lem:1}, and $\widehat{L}:=\widehat{\pi}^*L$ is numerically trivial on $\widehat{F}$. 

By \cite[Lemma 3.1]{LP18a} there exist a birational morphism $\pi\colon Y'\to Y$ from a smooth projective variety $Y'$, a nef $\Q$-divisor $L_{Y'}$ on $Y'$, a smooth projective variety $X'$ and a commutative diagram
\[
\xymatrix{ 
X' \ar[d]_{\varphi'} \ar[r]^{\pi'} & \widehat X  \ar[d]^{\widehat\varphi}\\
Y' \ar[r]_{\pi} & Y
}
\]
such that, if we set $L':=\pi'^*\widehat L$, then
$$L'\equiv \varphi'^*L_{Y'}.$$
We may write
$$K_{X'}+\Delta'\sim_\Q\pi'^*\big(K_{\widehat{X}}+\widehat\Delta\big)+E',$$
where $\Delta'$ and $E'$ are effective $\Q$-divisors on $X'$ without common components. Then if $F'$ is a general fibre of $\varphi'$, then $\nu\big(F',(K_{X'}+\Delta')|_{F'}\big)=0$ by Lemma \ref{lem:1}. Since
$$K_{X'}+\Delta'+tL'\sim_\Q\pi'^*\widehat{\pi}^*(K_X+\Delta+tL)+\pi'^*\widehat{E}+E'$$
for $t\geq0$, and since $X$ is $\Q$-factorial, it suffices to show that $K_{X'}+\Delta'+tL'$ is num-effective for any $t\geq0$. 

Therefore, by replacing $X$ by $X'$, $\Delta$ by $\Delta'$, $L$ by $L'$, and $\varphi$ by $\varphi'$, we achieve Assumption 1.

\medskip

\emph{Step 2.}
We now run any relative $(K_X+\Delta)$-MMP with scaling of an ample divisor over $Y$, and note that this MMP is $L$-trivial by Assumption 1. Since $\nu\big(X,(K_X+\Delta)|_F\big)=0$ for a general fibre $F$ of $\varphi$ by Assumption 1, the pair $(F,\Delta|_F)$ has a good model by \cite[Corollaire 3.4]{Dru11} and \cite[Corollary V.4.9]{Nak04}. Therefore, by \cite[Theorem 2.12]{HX13}, this MMP terminates with a relative good model $X_{\min}$ of $(X,\Delta)$ over $Y$. Denote by 
$$\theta\colon X\dashrightarrow X_{\min}$$ 
the corresponding birational morphism and by 
$$\varphi_{\min}\colon X_{\min}\to Y$$
the induced morphism to $Y$. Set $\Delta_{\min}:=\theta_*\Delta$ and $L_{\min}:=\theta_*L$. Let 
$$\tau\colon X_{\min}\to T$$ 
be the map to the relative canonical model $T$ over $Y$, with induced morphism $\varphi_T\colon T\to Y$, yielding the commutative diagram 
\[
\xymatrix{ 
X \ar[dr]_{\varphi} \ar@{-->}[rr]^{\theta} && X_{\min} \ar@/^0.5pc/[d]^{\tau} \ar[dl]^{\varphi_{\min}}\\
& Y\ar@{<-}[r]_{\varphi_T} & T.
}
\]

Notice that $\dim T<n$ by Assumption 1. Moreover, there exists a $\varphi_T$-ample $\Q$-divisor $A$ on $T$ such that $K_{X_{\min}}+\Delta_{\min}\sim_\Q\tau^*A$. By \cite[Theorem 0.2]{Amb05a} there exists an effective $\Q$-divisor $\Delta_T$ on $T$ such that the pair $(T,\Delta_T)$ is klt and such that 
$$K_{X_{\min}}+\Delta_{\min}\sim_\Q\tau^*(K_T+\Delta_T).$$ 
Therefore, since $\theta$ is $L$-trivial as noted above, for any $t\geq0$ we have
$$K_{X_{\min}}+\Delta_{\min}+tL_{\min}\sim_\Q\tau^*(K_T+\Delta_T+t\varphi_T^*L_Y).$$
By the assumptions of the theorem and since we assume Theorem \ref{thm:mainreduction} in dimension $n-1$, the Generalised Nonvanishing Conjecture holds in dimensions at most $n-1$. Therefore, for each $t\geq0$ the divisor $K_T+\Delta_T+t\varphi_T^*L_Y$ is num-effective, and therefore, the previous equation shows that  for every $t\geq0$ the divisor $K_{X_{\min}}+\Delta_{\min}+tL_{\min}$ is num-effective. But then each $K_X+\Delta+tL$ is num-effective by Lemma \ref{lem:MMPnumeff}.
\end{proof}

\subsection{An auxiliary result}

\begin{thm}\label{thm:Support}
Assume the termination of flips in dimensions at most $n-1$, and the Weak Generalised Nonvanishing Conjecture in dimension $n$. 

Let $(X,\Delta)$ be an $n$-dimensional projective $\Q$-factorial terminal pair such that $K_X+\Delta$ is pseudoeffective, and let $L$ be a nef $\Q$-divisor on $X$. Assume that there exists an effective $\Q$-divisor $D$ such that
$$K_X + \Delta \sim_\Q D \geq 0 \quad \text{and} \quad \Supp\Delta\subseteq\Supp D.$$
Then there exists a rational $t_0>0$ such that $K_X+\Delta+tL$ is num-effective for every $0<t\leq t_0$.
\end{thm}

\begin{proof}
We follow the arguments of \cite{DL15}. 

\medskip

\emph{Step 1.}
We need to show that there exists a rational $t_0>0$ such that
\begin{equation}\label{eq:304}
D+tL\text{ is num-effective for every }0<t\leq t_0.
\end{equation}
Let $f\colon X'\to X$ be a log resolution of the pair $(X,D)$. Then we may write
\begin{equation}\label{eq:305}
K_{X'}+\Delta_1\sim_\Q f^*(K_X+\Delta)+F',
\end{equation}
where $\Delta_1$ and $F'$ are effective $\Q$-divisors with no common components, and $\Delta_1=f^{-1}_*\Delta$ since $(X,\Delta)$ is a terminal pair. For a rational number $0<\varepsilon\ll1$, denote $\Delta_2:=\Delta_1+\varepsilon(f^*D+F')$ and $D_2:=(1+\varepsilon)(f^*D+F')$. Then the assumptions of the theorem and \eqref{eq:305} give
$$K_{X'}+\Delta_2\sim_\Q D_2\geq0\quad \text{with}\quad \Supp\Delta_2=\Supp D_2,$$
where the pair $(X',\Delta_2)$ is terminal and log smooth. Denote $L':=f^*L$. Then 
$$\textstyle D_2+tL'\sim_\Q (1+\varepsilon)f^*\big(D+\frac{t}{1+\varepsilon}L\big)+(1+\varepsilon)F'.$$
Therefore, since $F'$ is effective and $f$-exceptional, by \eqref{eq:304} it suffices to show that there exists a rational $t_2>0$ such that
\begin{equation}\label{eq:306}
D_2+tL'\text{ is num-effective for every }0<t\leq t_2.
\end{equation}
Set $\Delta':=\lceil\Delta_2\rceil$ and $D':=D_2+\lceil\Delta_2\rceil-\Delta_2$. Then
$$K_{X'}+\Delta'\sim_\Q D'\geq0\quad \text{with}\quad \Supp\Delta'=\Supp D',$$
where the pair $(X',\Delta')$ is log smooth and $\Delta'$ is a reduced divisor. By \eqref{eq:306} and by Lemma \ref{lem:easylemma}, it suffices to show that there exists a rational $t_0'$ such that
\begin{equation}\label{eq:307}
D'+tL'\text{ is num-effective for every }0<t\leq t_0'.
\end{equation}

\medskip

\emph{Step 2.}
Let $m$ be the smallest positive integer such that 
$$m(K_{X'}+\Delta')\sim mD',$$
and denote $G'=mD'$. Let $\pi\colon X''\to X'$ be the normalisation of the corresponding $m$-fold cyclic covering ramified along $G'$. The variety $X''$ is irreducible by \cite[Lemma 3.15(a)]{EV92} since $m$ is minimal. Then there exists an effective Cartier divisor $D_1''$ on $X''$ such that 
\begin{equation}\label{eq:2}
\pi^*G'=mD_1''\quad\text{and}\quad\pi^*(K_{X'}+\Delta')\sim D_1'',
\end{equation}
and let $\Delta_1''=(D_1'')_\textrm{red}$. By the Hurwitz formula, we have
\begin{equation}\label{eq:2a}
K_{X''}+\Delta_1''=\pi^*(K_{X'}+\Delta').
\end{equation}
Pick a rational number $0<\delta\ll1$ such that 
$$\Delta'':=\Delta_1''-\delta\pi^*\Delta'\geq0,\ D'':=D_1''-\delta\pi^*\Delta'\geq0\text{ and } \Supp D''=\Supp D_1''.$$
Then
$$K_{X''}+\Delta''=\pi^*\big(K_{X'}+(1-\delta)\Delta'\big)\quad\text{and}\quad K_{X''}+\Delta''\sim_\Q D'',$$
and therefore, the pair $(X'',\Delta'')$ is klt by \cite[Proposition 5.20]{KM98}. 
Denote $L'':=\pi^*L'$. By \eqref{eq:307} and by \cite[Lemma 2.15]{LP18a} it suffices to show that there exists a rational $t_1''>0$ such that
$$D_1''+tL''\text{ is num-effective for every }0<t\leq t_1''.$$
Since $\Supp D''=\Supp D_1''$, by Lemma \ref{lem:easylemma} it suffices to show that there exists a rational $t_0''>0$ such that
\begin{equation}\label{eq:308}
D''+tL''\text{ is num-effective for every }0<t\leq t_0''.
\end{equation}

\medskip

\emph{Step 3.}
We claim that $X''$ is not uniruled and that $\kappa(X'',K_{X''})\geq0$. Assuming the claim, we show first that it implies the theorem.

The assumptions of the theorem and \cite[Corollary 1.2]{HMX14} then imply that any sequence of flips of the pair $(X'',\Delta'')$ terminates. Therefore, by Proposition \ref{pro:contocon} there exists an $L''$-trivial $(K_{X''}+\Delta'')$-MMP $\varphi\colon X''\dashrightarrow Y$ such that $K_Y+\varphi_*\Delta''+m\varphi_*L''$ is nef for some fixed $m\gg0$. Denote $\Delta_Y:=\varphi_*\Delta''$ and $L_Y:=\varphi_*L''$. By \eqref{eq:308} and by Lemma \ref{lem:MMPnumeff} it suffices to show that there exists a rational $t_0^Y>0$ such that
\begin{equation}\label{eq:3c}
K_Y+\Delta_Y+tL_Y\text{ is num-effective for every }0<t\leq t_0^Y.
\end{equation}

As above, any sequence of flips of the pair $(Y,\Delta_Y)$ terminates. As in Step 1 and the beginning of Step 2 of the proof of \cite[Lemma 5.2]{LP18a}, there is a birational contraction
$$\theta\colon (Y,\Delta_Y)\dashrightarrow (Y_{\min},\Delta_{\min})$$ 
which is a composition of a sequence of operations of a $(K_Y+\Delta_Y)$-MMP, there is a positive rational number $\lambda$ and a divisor $L_{\min}:=\theta_*(mL_Y)$ on $Y_{\min}$ such that:
\begin{enumerate}
\item[(a)] $K_{Y_{\min}}+\Delta_{\min}$ is nef and $s(K_{Y_{\min}}+\Delta_{\min})+L_{\min}$ is nef for all $s\geq\lambda$,
\item[(b)] the map $\theta$ is $\big(s(K_Y+\Delta_Y)+mL)$-negative for $s>\lambda$.
\end{enumerate}

Denote $M:=\lambda(K_{Y_{\min}}+\Delta_{\min})+L_{\min}$; thus $M$ is nef by (a). By the Weak Generalised Nonvanishing Conjecture, the divisor 
$$K_{Y_{\min}}+\Delta_{\min}+tM=t\big((\lambda+1/t)(K_{Y_{\min}}+\Delta_{\min})+L_{\min}\big)$$ 
is num-effective for every $t>0$. Then by (b) and by Lemma \ref{lem:MMPnumeff}, the divisor
$$K_Y+\Delta_Y+tL_Y\quad\text{is num-effective}$$
for all $0<t\leq m/\lambda$. This proves (i) and finishes the proof of the theorem.

\medskip

\emph{Step 4.}
It remains to prove the claim from Step 3. Let $g\colon W\to X''$ be a log resolution of the pair $(X'',\Delta_1'')$. By \eqref{eq:2} and \eqref{eq:2a} we may write
\begin{equation}\label{eq:312}
K_W+\Delta_W\sim_\Q g^*(K_{X''}+\Delta_1'')+E_W\sim g^*D_1''+E_W,
\end{equation}
where $\Delta_W$ and $E_W$ are effective $\Q$-divisors with no common components. 

We claim that $\Delta_W$ and $E_W$ are \emph{integral} divisors. Let $E''$ be a geometric valuation over $X''$. Then by \cite[Proposition 2.14]{DL15}, there exists a geometric valuation $E'$ over $X'$ and an integer $1\leq r\leq m$ such that
$$a(E'',X'',\Delta_1'')+1=r\big(a(E',X',\Delta')+1\big).$$
Since $(X',\Delta')$ is log smooth and $\Delta'$ is reduced, we have $a(E',X',\Delta')\in\Z$, hence $a(E'',X'',\Delta_1'')\in\Z$, which proves the claim.

\medskip

Therefore, the divisor 
$$G_W:=g^*D_1''+E_W-\Delta_W$$ 
is Cartier, and by \eqref{eq:312} we have
$$K_W\sim_\Q G_W.$$
Then it suffices to show that
\begin{equation}\label{eq:309}
G_W\geq0, \quad\text{and in particular,}\quad \kappa(W,K_W)\geq0. 
\end{equation} 
Indeed, this implies that $W$ is not uniruled. Since $K_{X''}\sim_\Q g_*K_W$, the claim from Step 3 follows.

In order to show \eqref{eq:309}, fix a component $S''$ of $\Delta_W$. Then $a(S'',X'',\Delta_1'')=-1$. By \cite[Proposition 2.14]{DL15}, there exists a geometric valuation $S'$ over $X'$ and an integer $1\leq r\leq m$ such that $\pi\big(c_{X''}(S'')\big)=c_{X'}(S')$ and
$$a(S'',X'',\Delta_1'')+1=r\big(a(S',X',\Delta')+1\big).$$
Here, $c_{X''}(S'')$ denotes the center of $S''$ on $X''$, and analogously for $c_{X'}(S')$. This implies $a(S',X',\Delta')=-1$, thus $c_{X'}(S')\subseteq\Supp\Delta'$ because $(X',\Delta')$ is log smooth. From here we obtain $c_{X''}(S'')\subseteq\pi^{-1}(\Supp\Delta')=\Supp D_1''$, and in particular, $S''\subseteq\Supp g^*D_1''$. Therefore, 
$$\mult_{S''}g^*D_1''\geq1=\mult_{S''}\Delta_W.$$
Now \eqref{eq:309} follows by the definition of $G_W$. This finishes the proof.
\end{proof}

\subsection{Non-pseudoeffective canonical class}

\begin{lem}\label{lem:Knotpsef}
Assume the termination of flips in dimensions at most $n-1$, the Abundance Conjecture in dimensions at most $n-1$, the existence of minimal models in dimension $n$, and the Weak Generalised Nonvanishing Conjecture in dimension $n$. 

Assume Theorem \ref{thm:B} in dimension $n$. Let $(X,\Delta)$ be an $n$-dimensional projective $\Q$-factorial terminal pair such that $K_X+\Delta$ is pseudoeffective, and let $L$ be a nef $\Q$-divisor on $X$. Assume that $K_X$ is not pseudoeffective.

Then $\kappa(X,K_X+\Delta)\geq0$ and there exists a rational $t_0>0$ such that $K_X+\Delta+tL$ is num-effective for every $0< t\leq t_0$.
\end{lem}

\begin{proof}
Define
$$\tau:=\min\{t\in\R\mid K_X+t\Delta\text{ is pseudoeffective}\}.$$
Observe that $0<\tau\leq1$ by our assumptions. Then by \cite[Theorem 3.3]{DL15} there exists a $\Q$-divisor $D_\tau\geq0$ such that $K_X+\tau\Delta\sim_\Q D_\tau$. This yields 
$$K_X + \Delta \sim_\Q D \geq 0 , \quad \text{where} \quad D =D _\tau + (1-\tau )\Delta.$$
Pick a rational number $0<\mu\ll1$ such that the pair $\big(X,(1+\mu)\Delta\big)$ is terminal, and denote $\Delta':=(1+\mu)\Delta$. Observe that
$$K_X+\Delta'\sim_\Q D+\mu\Delta\quad\text{and}\quad \Supp\Delta'\subseteq\Supp(D+\mu\Delta).$$
Then by Theorem \ref{thm:Support} there exist a rational number $\lambda>0$ and an effective $\Q$-divisor $G$ such that 
\begin{equation}\label{eq:121a}
K_X+(1+\mu)\Delta+\lambda L\equiv G. 
\end{equation} 

Pick a rational number $0<\delta\ll1$ such that the pair $\big(X,(1-\delta)\Delta+\frac{\delta}{\mu}G\big)$ is terminal, and denote $\Delta_X:=(1-\delta)\Delta+\frac{\delta}{\mu}G$. Then by \eqref{eq:121a} we have
$$\textstyle K_X+\Delta_X\equiv \big(1+\frac{\delta}{\mu}\big)\big(K_X+\Delta+\frac{\delta\lambda}{\delta+\mu}L\big).$$
In particular, $K_X+\Delta_X$ is pseudoeffective. Since then $\kappa(X,K_X+\Delta_X)\geq0$ by Lemma \ref{lem:termination}, the result follows.
\end{proof} 

\subsection{Proofs of Theorems \ref{thm:mainreduction} and \ref{thm:B}}\label{subsec:mainreduction}

We finally have all the ingredients to prove Theorems \ref{thm:mainreduction} and \ref{thm:B}.

\begin{thm}\label{thm:induction2}
Assume the termination of flips in dimensions at most $n-1$, the Abundance Conjecture in dimensions at most $n-1$, the existence of minimal models in dimension $n$, and the Weak Generalised Nonvanishing Conjecture in dimension $n$. 

Then Theorem \ref{thm:B} in dimension $n$ implies Theorem \ref{thm:mainreduction} in dimension $n$.
\end{thm}

\begin{proof}
Let $(X,\Delta)$ be an $n$-dimensional projective klt pair such that $K_X+\Delta$ is pseudoeffective, and let $L$ be a nef $\Q$-divisor on $X$. By Proposition \ref{pro:reduction} and by Lemma \ref{lem:MMPnumeff} we may assume that $K_X+\Delta+L$ is nef.
 
\medskip

\emph{Step 1.}
If $n(X,K_X+\Delta+L)<n$, then we conclude by Theorem \ref{thm:B}.

\medskip

\emph{Step 2.}
Therefore, from now on we may assume that
$$n(X,K_X+\Delta+L)=n.$$
We first prove that in this case part (i) of Theorem \ref{thm:mainreduction} implies part (ii) of Theorem \ref{thm:mainreduction}. Indeed, then there exists an effective $\Q$-divisor $N$ such that 
$$K_X+\Delta+t_0L\equiv N.$$ 
Pick a small positive rational number $\varepsilon$ such that $(X,\Delta+\varepsilon N)$ is klt. We have 
$$\kappa(X,K_X+\Delta+\varepsilon N)\geq0$$
by (i), hence by \cite[Corollary 1.2]{HMX14} there exists a $(K_X+\Delta+\varepsilon N)$-MMP 
$$\varphi\colon X\dashrightarrow Y$$ 
which terminates. The divisor $K_Y+\varphi_*\Delta+\varepsilon \varphi_*N$ is semiample, since we are assuming the semiampleness part of the Abundance Conjecture in dimension $n$. Since $n(X,K_X+\Delta+L)=n$, we have $$(K_X+\Delta+\varepsilon N)\cdot C>0$$ for every curve $C$ on $X$ passing through a very general point on $X$ by Lemma \ref{lem:nefredsum}, and therefore $$n(Y,K_Y+\varphi_*\Delta+\varepsilon\varphi_*N)=n$$ by Corollary \ref{cor:nefdim}. But then $K_Y+\varphi_*\Delta+\varepsilon\varphi_*N$ is big, hence $K_X+\Delta+\varepsilon N$ is also big. As 
$$K_X+\Delta+\varepsilon N\equiv(1+\varepsilon)(K_X+\Delta)+\varepsilon t_0 L,$$
and $K_X+\Delta$ and $L$ are both pseudoeffective, it is easy to see that $K_X+\Delta+tL$ is big for all $t>0$.

\medskip

\emph{Step 3.}
It remains to prove (i). In this step we show that we may additionally assume the following:

\medskip

\emph{Assumption 1.} The pair $(X,\Delta)$ is $\Q$-factorial and terminal, and $K_X+\Delta$ is nef.

\medskip

To this end, let $\pi\colon \widetilde X\to X$ be a small $\Q$-factorialisation. As $K_{\widetilde X}+\pi_*^{-1}\Delta\sim_\Q\pi^*(K_X+\Delta)$, by replacing $X$ by $\widetilde X$, $\Delta$ by $\pi_*^{-1}\Delta$ and $L$ by $\pi^*L$, we may assume that $X$ is $\Q$-factorial by \cite[Lemma 2.14]{LP18a}.

By Lemma \ref{lem:termination}, we have 
$$\kappa(X,K_X+\Delta)\geq0$$
and any sequence of flips of the pair $(X,\Delta)$ terminates. Then as in Step 1 and the beginning of Step 2 of the proof of \cite[Lemma 5.2]{LP18a}, there is a birational contraction
$$\theta\colon (X,\Delta)\dashrightarrow (X_{\min},\Delta_{\min})$$ 
which is a composition of a sequence of operations of a $(K_X+\Delta)$-MMP, a positive rational number $\lambda$ and a divisor $L_{\min}:=\theta_*L$ on $X_{\min}$ such that:
\begin{enumerate}
\item[(a)] $K_{X_{\min}}+\Delta_{\min}$ is nef and $s(K_{X_{\min}}+\Delta_{\min})+L_{\min}$ is nef for all $s\geq\lambda$,
\item[(b)] the map $\theta$ is $\big(s(K_X+\Delta)+L)$-negative for $s>\lambda$,
\item[(c)] $(K_{X_{\min}}+\Delta_{\min}+tL_{\min})\cdot C>0$ for every curve $C$ on $X_{\min}$ passing through a very general point on $X_{\min}$ and for every $t>0$.
\end{enumerate}

Denote 
$$M:=\lambda(K_{X_{\min}}+\Delta_{\min})+L_{\min}.$$ 
Then $M$ is nef by (a) and 
$$n(X_{\min},K_{X_{\min}}+\Delta_{\min}+tM)=n\quad\text{for all }t>0$$
by (c) and by Corollary \ref{cor:nefdim}. Assume that there exists a rational $t_0^{\min}>0$ such that $K_{X_{\min}}+\Delta_{\min}+tM$ is num-effective for every $0<t\leq t_0^{\min}$. Since
$$\textstyle K_{X_{\min}}+\Delta_{\min}+tM=(t\lambda+1)\big(K_{X_{\min}}+\Delta_{\min}+\frac{t}{t\lambda+1} L_{\min}\big),$$ 
the divisor $K_{X_{\min}}+\Delta_{\min}+t L_{\min}$ is num-effective for every $0<t\leq \frac{t_0^{\min}}{t_0^{\min}\lambda+1}$. But then by (b) and by Lemma \ref{lem:MMPnumeff}, the divisor $K_X+\Delta+t L$ is num-effective for every $0<t\leq \frac{t_0^{\min}}{t_0^{\min}\lambda+1}$. 

Therefore, by replacing $X$ by $X_{\min}$, $\Delta$ by $\Delta_{\min}$ and $L$ by $M$, we achieve that the variety $X$ is $\Q$-factorial, and the divisor $K_X+\Delta$ is nef.

Now, let 
$$\xi\colon (X',\Delta')\to (X,\Delta)$$ 
be the composition of a terminalisation and a small $\Q$-factorialisation. Since $K_{X'}+\Delta'\sim_\Q\xi^*(K_X+\Delta)$, it follows that, by replacing $X$ by $X'$, $\Delta$ by $\Delta'$ and $L$ by $\xi^*L$, we may additionally assume that $(X,\Delta)$ is a terminal pair, hence we achieve Assumption 1.

\medskip

\emph{Step 4.}
Now, if $K_X$ is pseudoeffective, then $X$ is not uniruled and Theorem \ref{thm:mainreduction}(i) follows from the Weak Generalised Nonvanishing Conjecture. If $K_X$ is not pseudoeffective, then Theorem \ref{thm:mainreduction}(i) follows from Lemma \ref{lem:Knotpsef}. This completes the proof.
\end{proof}

Theorems \ref{thm:mainreduction} and \ref{thm:B} follow immediately from Theorems \ref{thm:induction} and \ref{thm:induction2}.

\section{A criterion for Nonvanishing}\label{sec:MMP}

In this section we prove a criterion -- Theorem \ref{thm:nonvanishingForms} -- for effectivity of divisors of the form $K_X+\Delta+L$, where $(X,\Delta)$ is a klt pair such that $K_X+\Delta$ is nef and $X$ is not uniruled, and if $L$ is a nef $\Q$-divisor on $X$. It is crucial for the remainder of the paper.

Theorem \ref{thm:nonvanishingForms} contains \cite[Theorem 4.3 and Theorem 8.1]{LP18} as special cases. 

\begin{thm} \label{thm:nonvanishingForms}
Assume the existence of good models for klt pairs in dimensions at most $n-1$. 

Let $(X,\Delta)$ be a projective terminal pair of dimension $n$ such that $K_X$ pseudoeffective and $K_X+\Delta$ is nef. Let $L$ be a nef Cartier divisor on $X$ and let $t$ be a positive integer such that $t(K_X+\Delta)$ is Cartier. Denote $M=t(K_X+\Delta)+L$, and let $\pi\colon Y\to X$ be a resolution of $X$. Assume that for some positive integer $p$ we have  
$$ H^0\big(Y,(\Omega^1_Y)^{\otimes p} \otimes \OO_Y(m\pi^*M)\big) \neq 0$$
for infinitely many integers $m$. Then $\kappa (X,M) \geq 0$. 
\end{thm} 

\begin{proof} 
Let $\rho\colon X'\to X$ be a small $\Q$-factorialisation of $X$. Then the pair $(X',\rho^{-1}_*\Delta)$ is terminal and $K_{X'}+\rho^{-1}_*\Delta\sim_\Q\rho^*(K_X+\Delta)$. We may assume that $\pi$ factors through $\rho$. By replacing $(X,\Delta)$ by $(X',\rho^{-1}_*\Delta)$ and $L$ by $\rho^*L$, we may thus assume that $X$ is $\Q$-factorial. 

\medskip

If $M\equiv0$, then the assumptions imply that $K_X\equiv0$ (hence $K_X\sim_\Q0$ by \cite[Theorem 8.2]{Kaw85b}), $\Delta=0$ and $L\equiv0$. Then $\kappa(X,M)=\kappa(X,L)\geq0$ by the first part of the proof of \cite[Theorem 8.1]{LP18}.

\medskip

Therefore, from now on we may assume that $M\not\equiv0$. We apply \cite[Lemma 4.1]{LP18} with $\mathcal E := (\Omega^1_Y)^{\otimes p} $ and $\mathcal L := \pi^*\OO_X(M)$. Then there exist a positive integer $r$, a saturated line bundle $\mathcal M$ in $\bigwedge^r\mathcal E$, an infinite set $\mathcal S\subseteq\N$ and integral divisors $N_m\geq0$ for $m\in\mathcal S$ such that 
\begin{equation}\label{eq:infmany}
\OO_Y(N_m) \simeq \mathcal M\otimes \mathcal L^{\otimes m}.
\end{equation}
Since $X$ is terminal and $K_X$ is pseudoeffective, the divisor $K_Y$ is also pseudoeffective, hence \cite[Proposition 4.2]{LP18} implies that there exist a positive integer $\ell$ and a pseudoeffective divisor $F$ such that
\begin{equation}\label{eq:relation}
N_m+ F \sim m\pi^*M+\ell K_Y.
\end{equation}
Noting that $\pi_*N_m$ is effective and that $\pi_*F$ is pseudoeffective, by pushing forward the relation \eqref{eq:relation} to $X$ we get
$$\pi_*N_m+\pi_*F\sim_\Q mM +\ell K_X,$$ 
and hence
$$\pi_*N_m+(\pi_*F+\ell\Delta)\sim_\Q mM +\ell (K_X+\Delta).$$ 
Now we conclude by Theorem \ref{thm:MMPtwist}.
\end{proof} 

\begin{thm} \label{thm:MMPtwist}
Assume the existence of good models for klt pairs in dimensions at most $n-1$. 

Let $(X,\Delta)$ be a $\Q$-factorial projective klt pair of dimension $n$ such that $K_X+\Delta$ is nef, and let $L$ be a nef Cartier divisor on $X$. Let $t$ be a positive integer such that $t(K_X+\Delta)$ is Cartier, and denote $M=t(K_X+\Delta)+L$. Assume that there exist a pseudoeffective $\Q$-divisor $F$ on $X$, a positive integer $\ell$ and an infinite subset $\mathcal S\subseteq\N$ such that 
\begin{equation}\label{eq:rel2a}
N_m+F\sim_\Q mM+\ell(K_X+\Delta)
\end{equation}
for all $m\in\mathcal S$, where $N_m\geq0$ are integral Weil divisors. Then 
$$\kappa(X,M)=\max\{\kappa(X,N_m)\mid m\in\mathcal S\}\geq0.$$
\end{thm} 

\begin{proof}
Steps 1--4 follow closely the proof of \cite[Theorem 3.1]{LP18}, and we include the details for the benefit of the reader.

\medskip

Note first that \eqref{eq:rel2a} implies
\begin{equation}\label{eq:rel2b}
N_p-N_q\sim_\Q (p-q)M\quad\text{for all }p,q\in\mathcal S.
\end{equation}
It suffices to show that
\begin{equation}\label{eq:kappalargem}
\kappa(X,M)\geq\kappa(X,N_m)\quad\text{for all large }m\in\mathcal S.
\end{equation}
Indeed, then $\kappa(X,M)\geq0$, hence \eqref{eq:rel2b} gives $\kappa(X,M)\leq\kappa(X,N_m)$ for all large $m\in\mathcal S$ and $\kappa(X,N_q)\leq \kappa(X,N_p)$ for $p,q\in\mathcal S$ with $q<p$, which implies the theorem.

\medskip

\emph{Step 1.}
We claim that for every $m\in\mathcal S$ and every rational number $\lambda>0$ we have
\begin{equation}\label{eq:kodaira}
\kappa(X,M+\lambda N_m)\geq\kappa(X,N_m),
\end{equation}
and in particular,
\begin{equation}\label{eq:kodaira3}
\kappa(X,M+\lambda N_m)\geq0 \text{ for every }m\in\mathcal S \text{ and every rational }\lambda>0.
\end{equation}
Indeed, fix $m\in\mathcal S$ and a rational number $\lambda>0$. Pick $m'\in\mathcal S$ so that $\frac{1}{m'-m}<\lambda$. Then by \eqref{eq:rel2b} we have
\begin{align}\label{eq:33a}
M+\lambda N_m& \textstyle =\frac{1}{m'-m}\big((m'-m)M+N_m\big)+\big(\lambda-\frac{1}{m'-m}\big)N_m\\
&\textstyle \sim_\Q \frac{1}{m'-m} N_{m'}+\big(\lambda-\frac{1}{m'-m}\big)N_m,\notag
\end{align}
which proves \eqref{eq:kodaira}.

\medskip

There are now three cases to consider.

\medskip

\emph{Step 2.}
First assume that 
$$M+\lambda N_p \text{ is big for some }p\in\mathcal S \text{ and some rational number }\lambda>0.$$
Since 
$$\textstyle M \sim_\Q \frac{1}{1+\lambda p+\lambda \ell/t} \big(M+\lambda N_p+\lambda F+\frac{\lambda \ell}{t} L\big)$$
by \eqref{eq:rel2a}, this implies that $M$ is big, which shows \eqref{eq:kappalargem}.

\medskip

\emph{Step 3.}
Now assume that 
$$\kappa(X,M+\lambda N_p)=0 \text{ for some }p\in\mathcal S \text{ and some rational number }\lambda>0.$$
Fix $q\in\mathcal S$ such that $0<\frac{1}{q-p}<\lambda$. Then as in \eqref{eq:33a} we have
$$\textstyle M+\lambda N_p\sim_\Q \frac{1}{q-p} N_q+\big(\lambda-\frac{1}{q-p}\big)N_p,$$
hence $\kappa(X,N_q)\leq\kappa(X,M+\lambda N_p)=0$, and therefore $\kappa(X,N_q)=0$. Let $r\in\mathcal S$ be such that $r>q$. Then by \eqref{eq:rel2b} we have
$$M\sim_\Q\frac{1}{q-p}(N_q-N_p)\quad\text{and}\quad M\sim_\Q \frac{1}{r-p}(N_r-N_p),$$
so that
$$(r-p)N_q\sim_\Q(q-p)N_r+(r-q)N_p\geq0.$$
Since $\kappa(X,N_q)=0$, this implies
$$(r-p)N_q=(q-p)N_r+(r-q)N_p, $$
and hence 
\begin{equation}\label{eq:kappaSupp}
\Supp N_r\subseteq\Supp N_q \quad\text{and}\quad \kappa(X,N_r)=0.
\end{equation}
Therefore, for $r>q$, all divisors $N_r$ are supported on the reduced Weil divisor $\Supp N_q$. By \cite[Lemma 2.1]{LP18} there are positive integers $p'\neq q'$ larger than $q$ in $\mathcal S$ such that $N_{p'}\leq N_{q'}$, and thus by \eqref{eq:rel2b},
$$(q'-p')M\sim_\Q N_{q'}-N_{p'}\geq0,$$
hence $\kappa(X,M)\geq 0$. This together with \eqref{eq:kappaSupp} shows \eqref{eq:kappalargem}.

\medskip

\emph{Step 4.}
Finally, by \eqref{eq:kodaira3} and by Steps 2 and 3 for the rest of the proof we may assume that
\begin{equation}\label{eq:kodaira1}
0<\kappa(X,M+\lambda N_p)<n \quad\text{for every }p\in\mathcal S\text{ and every }\lambda>0.
\end{equation}
In this step we show that 
\begin{equation}\label{eq:positivekodaira}
\kappa(X,K_X+\Delta)\geq0.
\end{equation}
Indeed, fix $p\in\mathcal S$ and denote $P=M+N_p$. Then $\kappa(X,P)\in\{1,\dots,n-1\}$ by \eqref{eq:kodaira1}, hence by \cite[Lemma 2.3]{LP18} there exists a resolution $\pi\colon Y\to X$ and a fibration $f\colon Y\to Z$:
\[
\xymatrix{ 
Y \ar[d]_{\pi} \ar[r]^{f} & Z\\
X & 
}
\]
such that $\dim Z=\kappa(X,P)$, and for a very general fibre $F$ of $f$ and for every $\pi$-exceptional $\Q$-divisor $G$ on $Y$ we have
\begin{equation}\label{eq:exceptional}
\kappa\big(F,(\pi^*P+G)|_F\big)=0.
\end{equation} 
There exist effective $\Q$-divisors $\Delta_Y$ and $E$ without common components such that
$$K_Y+\Delta_Y\sim_\Q\pi^*(K_X+\Delta)+E,$$
and it is enough to show that $\kappa(Y,K_Y+\Delta_Y)\geq0$. By \eqref{eq:exceptional} we have
$$\kappa\big(F,\big(t(K_Y+\Delta_Y)+\pi^*L+\pi^*N_p\big)|_F\big)=\kappa\big(F,(\pi^*P+tE)|_F\big)=0.$$
In particular, $(K_Y+\Delta_Y)|_F$ is not big on $F$ since otherwise $\big(t(K_Y+\Delta_Y)+\pi^*L+\pi^*N_p\big)|_F$ would also be big on $F$. But then $\kappa(Y,K_Y+\Delta_Y)\geq0$ by \cite[Proposition 3.2]{LP18}, which proves the claim \eqref{eq:positivekodaira}.

\medskip

\emph{Step 5.}
From \eqref{eq:kodaira1} and \eqref{eq:positivekodaira} it immediately follows that 
$$0\leq\kappa(X,K_X+\Delta+\lambda N_p)<n\quad \text{ for every }p\in\mathcal S\text{ and every }\lambda>0.$$
Assume that there exist two elements $p<q$ in $\mathcal S$ and positive rational numbers $\lambda_p$ and $\lambda_q$ such that 
$$\kappa(X,K_X+\Delta+\lambda_p N_p)=\kappa(X,K_X+\Delta+\lambda_q N_q)=0.$$
Then \eqref{eq:positivekodaira} implies $\kappa(X,N_p)=\kappa(X,N_q)=0$. Thus, as in Step 3 we have $\kappa(X,M)\geq0$ and $\kappa(X,N_r)=0$ for all $r\in\mathcal S$ with $r\geq q$, which yields \eqref{eq:kappalargem}.

\medskip

\emph{Step 6.} 
Therefore, by Step 5 we may assume that
\begin{equation}\label{eq:kodaira2}
0<\kappa(X,K_X+\Delta+\lambda N_p)<n \quad\text{for every }p\in\mathcal S\text{ and every }\lambda>0.
\end{equation}
Fix integers $w>k$ in $\mathcal S$ and a rational number $0<\varepsilon\ll1$ such that
\begin{enumerate}
\item[(a)] the pair $(X,\Delta+\varepsilon N_k)$ is klt, 
\item[(b)] $\varepsilon(w-k)>2n$.
\end{enumerate}

Then in Steps 6--8 we show that for every $s\in\mathcal S$ with $s>w$ we have
\begin{equation}\label{eq:finalclaim}
\kappa(X,M)\geq\kappa(X,N_s),
\end{equation}
which then implies \eqref{eq:kappalargem} and proves the theorem.

\medskip

Fix now any $s\in\mathcal S$ with $s>w$, and fix a rational number $0<\delta<\varepsilon$ such that:
\begin{enumerate}
\item[(c)] the pairs $(X,\Delta+\delta N_w)$ and $(X,\Delta+\delta N_s)$ are klt.
\end{enumerate}
Since we are assuming the existence of good models for klt pairs in lower dimensions, by \eqref{eq:kodaira2} and by \cite[Theorem 2.5]{LP18} applied to a birational model of the Iitaka fibration of $K_X+\Delta+\delta N_w$, there exists a $(K_X+\Delta+\delta N_w)$-MMP $\theta\colon X\dashrightarrow X_{\min}$ which terminates with a good model for $(X,\Delta+\delta N_w)$. 

\medskip

We claim that:  
\begin{enumerate}
\item[(i)] $\theta$ is $N_w$-negative,
\item[(ii)] $\theta$ is $(K_X+\Delta)$-trivial as well as $L$-trivial; hence, the proper transforms of $t(K_X+\Delta)$ and $L$ at every step of this MMP are nef Cartier divisors by \cite[Theorem 3.7(4)]{KM98},
\item[(iii)] $\theta$ is $(K_X+\Delta+\delta N_s)$-negative.
\end{enumerate} 

Indeed, it is enough to show the claim for the first step of the MMP, as the rest is analogous; one important point is that relations \eqref{eq:rel2a} and \eqref{eq:rel2b} also continue to hold for the proper transforms of the divisors involved. 

Let $c_R\colon X\to Z$ be the contraction of a $(K_X+\Delta+\delta N_w)$-negative extremal ray $R$ in this MMP. Since $K_X+\Delta$ is nef, we immediately get that $R$ is $N_w$-negative, which gives (i). 

For (ii), since $K_X+\Delta$ and $L$ are both nef, it suffices to show that $c_R$ is $M$-trivial. By the boundedness of extremal rays \cite[Theorem 1]{Kaw91}, there exists a rational curve $C$ contracted by $c_R$ such that 
$$(K_X+\Delta+\varepsilon N_k)\cdot C\geq {-}2n.$$
If $c_R$ were not $M$-trivial, then $M\cdot C\geq1$ as $M$ is nef and Cartier. But then \eqref{eq:rel2b} and the condition (b) above yield
$$(K_X+\Delta+\varepsilon N_w)\cdot C=(K_X+\Delta+\varepsilon N_k)\cdot C+\varepsilon(w-k)M\cdot C>0.$$
On the other hand, by (i) we have
$$(K_X+\Delta+\varepsilon N_w)\cdot C=(K_X+\Delta+\delta N_w)\cdot C+(\varepsilon-\delta)N_w\cdot C<0,$$
a contradiction which shows (ii). 

Finally, by \eqref{eq:rel2b} and by (ii) we have
$$(K_X+\Delta+\delta N_s)\cdot R=(K_X+\Delta+\delta N_w)\cdot R+\delta(s-w)M\cdot R<0,$$
which gives (iii).

\medskip

\emph{Step 7.}
Now, by (ii) we have
$$\kappa(X,M)=\kappa(X_{\min},\theta_*M),$$
the divisor $\theta_*F$ is pseudoeffective, and
$$\kappa(X_{\min},\theta_*N_m)\geq\kappa(X,N_m)\quad\text{for }m\in\mathcal S$$
by \cite[Lemma 2.8]{LP18}. 
Furthermore, $K_{X_{\min}}+\theta_*\Delta+\delta \theta_*N_w$ is semiample, and by \eqref{eq:rel2b} we have 
$$K_{X_{\min}}+\theta_*\Delta+\delta \theta_*N_s\sim_\Q K_{X_{\min}}+\theta_*\Delta+\delta \theta_*N_w+\delta(s-w)\theta_*M,$$
hence $K_{X_{\min}}+\theta_*\Delta+\delta \theta_*N_s$ is likewise nef by (ii). By (iii) and by \eqref{eq:kodaira2} we have
$$\kappa(X_{\min},K_{X_{\min}}+\theta_*\Delta+\delta\theta_*N_s)=\kappa(X,K_X+\Delta+\delta N_s)\in\{1,\dots,n-1\},$$
hence $K_{X_{\min}}+\theta_*\Delta+\delta\theta_*N_s$ is semiample by \cite[Theorem 2.5]{LP18} applied to a birational model of the Iitaka fibration of $K_{X_{\min}}+\theta_*\Delta+\delta\theta_*N_s$.
 
Therefore, in order to show \eqref{eq:kappalargem}, by replacing $X$ by $X_{\min}$, $\Delta$ by $\theta_*\Delta$, $L$ by $\theta_*L$, $F$ by $\theta_*F$ and $N_m$ by $\theta_*N_m$ for every $m\in\mathcal S$, we may additionally assume:

\medskip

\emph{Assumption 1.} The divisors $K_X+\Delta+\delta N_w$ and $K_X+\Delta+\delta N_s$ are semiample, and $\kappa(X,K_X+\Delta+\delta N_s)\in\{1,\dots,n-1\}$. 

\medskip

\emph{Step 8.}
Let $\varphi_w\colon X\to S_w$ and $\varphi_s\colon X\to S_s$ be the Iitaka fibrations associated to $K_X+\Delta+\delta N_w$ and $K_X+\Delta+\delta N_s$, respectively. Then there exist ample $\Q$-divisors $A_w$ on $S_w$ and $A_s$ on $S_s$ such that
\begin{equation}\label{eq:pullbacks}
K_X+\Delta+\delta N_w\sim_\Q\varphi_w^*A_w\quad\text{and}\quad K_X+\Delta+\delta N_s\sim_\Q\varphi_s^*A_s.
\end{equation}
If $\xi$ is a curve on $X$ contracted by $\varphi_s$, then by \eqref{eq:rel2b} we have
\begin{align*}
0&=(K_X+\Delta+\delta N_s)\cdot \xi=(K_X+\Delta+\delta N_w)\cdot \xi+\delta(s-w)M\cdot \xi\\
& =(K_X+\Delta+\delta N_w)\cdot \xi+\delta(s-w)(K_X+\Delta)\cdot \xi+\delta(s-w)L\cdot \xi,
\end{align*}
hence 
$$(K_X+\Delta+\delta N_w)\cdot \xi=(K_X+\Delta)\cdot \xi=L\cdot \xi=0$$
since $K_X+\Delta+\delta N_w$, $K_X+\Delta$ and $L$ are nef.

This implies two things: first, that $K_X+\Delta$ is $\varphi_s$-numerically trivial, and hence by Lemma \ref{lem:numlintrivial} there exists a $\Q$-Cartier $\Q$-divisor $B_s$ on $S_s$ such that
\begin{equation}\label{eq:pullbackklt}
K_X+\Delta\sim_\Q\varphi_s^*B_s.
\end{equation}

Second, every curve contracted by $\varphi_s$ is also contracted by $\varphi_w$. Hence, by the Rigidity lemma \cite[Lemma 1.15]{Deb01} there exists a morphism $\psi\colon S_s\to S_w$ such that $\varphi_w=\psi\circ\varphi_s$. 
\[
\xymatrix{ 
& X \ar[ld]_{\varphi_s} \ar[dr]^{\varphi_w} & \\
S_s \ar[rr]^{\psi} & & S_w 
}
\]
Therefore, denoting 
$$\textstyle L_s:=\frac{1}{\delta(s-w)} (A_s-\psi^*A_w)-tB_s$$
and
$$\textstyle F_s:=(\ell+\frac{1}{\delta}) B_s+\frac{1}{\delta(s-w)}(w A_s-s\psi^* A_w),$$
by \eqref{eq:rel2a}, \eqref{eq:rel2b}, \eqref{eq:pullbacks} and \eqref{eq:pullbackklt} we have 
\begin{equation}\label{eq:twopullbacks}
L\sim_\Q \varphi_s^*L_s\quad\text{and}\quad F\sim_\Q \varphi_s^*F_s.
\end{equation}
This implies that $L_s$ is nef and $F_s$ is pseudoeffective. Now by \eqref{eq:rel2a}, \eqref{eq:pullbacks} and \eqref{eq:twopullbacks} we have
\begin{align*}
(1+\delta\ell+\delta st)M & =t(K_X+\Delta+\delta N_s)+\delta t F+(1+\delta\ell)L\\
& \sim_\Q\varphi_s^*\big(tA_s+\delta t F_s+(1+\delta\ell)L_s\big).
\end{align*}
Finally, since $tA_s+\delta t F_s+(1+\delta\ell)L_s$ is a big divisor on $S_s$, we obtain
\begin{align*}
\kappa(X,M)&=\kappa\big(S_s,tA_s+\delta t F_s+(1+\delta\ell)L_s\big)=\dim S_s\\
&=\kappa(X,K_X+\Delta+\delta N_s)\geq\kappa(X,N_s),
\end{align*}
where the last inequality follows by \eqref{eq:positivekodaira}. This shows \eqref{eq:finalclaim} and finishes the proof.
\end{proof}

\section{Metrics with generalised algebraic singularities, I}\label{sec:Algsing} 

The goal of this section is to apply Theorem \ref{thm:nonvanishingForms} in the case when one knows that the nef divisors involved possess metrics with generalised algebraic singularities. This will have surprising consequences for the relationship between effectivity and num-effectivity. 

Additionally, we assume that $K_X$ is pseudoeffective and $K_X + \Delta$ is nef. The general case, where $K_X$ is not necessarily pseudoeffective and $K_X + \Delta$ is pseudoeffective rather then nef, will be treated in the following section. 

We start with the following result. Part (i) contains \cite[Corollary 4.5 and Corollary 8.5]{LP18} as special cases.

\begin{thm}\label{thm:nv0}
Assume the existence of good models for klt pairs in dimensions at most $n-1$. 

Let $(X,\Delta)$ be a projective terminal pair of dimension $n$ such that $K_X$ is pseudoeffective and $K_X+\Delta$ is nef, and let $L$ be a nef $\Q$-divisor on $X$. Assume that $\chi(X,\OO_X)\neq0$.
\begin{enumerate}
\item[(i)] Suppose that the divisor $K_X+\Delta+L$ has a singular metric with generalised algebraic singularities and semipositive curvature current. Then for every $\Q$-divisor $L'$ with $L\equiv L'$ we have $\kappa(X,K_X+\Delta+L') \geq 0$. 
\item[(ii)] In particular, if $\kappa(X,K_X+\Delta+L)\geq0$, then for every $\Q$-divisor $D$ with $K_X+\Delta+L\equiv D$ we have $\kappa(X,D)\geq0$. 
\end{enumerate}
\end{thm}

\begin{proof}
Part (ii) follows immediately from part (i).

The assumptions imply that $K_X+\Delta+L'$ has a singular metric with generalised algebraic singularities and semipositive curvature current. Therefore, we may assume that $L=L'$.

Let $\rho\colon X'\to X$ be a $\Q$-factorialisation of $X$. Then $\rho$ is an isomorphism in codimension $1$, hence 
$$K_{X'}+\rho_*^{-1}\Delta=\rho^*(K_X+\Delta)$$
and the pair $(X',\rho_*^{-1}\Delta)$ is terminal. By replacing $(X,\Delta)$ by $(X',\rho_*^{-1}\Delta)$, we may thus assume that $X$ is $\Q$-factorial. 

Choose a positive integer $\ell$ such that $M:=\ell(K_X+\Delta+L)$ is Cartier, and so that there exist a resolution $\pi\colon Y\to X$  and a metric $h$ with generalised algebraic singularities on $\pi^*\OO_X(M)$ as in Section \ref{sec:prelim}. Then the corresponding curvature current $T$ of $h$ has the Siu decomposition
$$T = \Theta+\sum_{j=1}^r \lambda_j D_j,$$
where the semipositive current $\Theta$ has all Lelong numbers zero, the numbers $\lambda_j$ are positive and rational and the divisor $\sum D_j$ has simple normal crossings. We have
\begin{equation}\label{eq:metric1}
\textstyle\mathcal I(h^{\otimes m})=\OO_Y\big({-}\sum_{j=1}^r\lfloor m\lambda_j\rfloor D_j\big).
\end{equation}

Arguing by contradiction, we assume that $\kappa(X,M)=-\infty$. Then by Theorem \ref{thm:nonvanishingForms}, for all $p\geq 0$ and for all $m\gg0$ we have
$$ H^0\big(Y,\Omega^p_Y \otimes \pi^*\OO_X(mM)\big)=0,$$
and thus
$$H^0\big(Y,\Omega^p_Y\otimes \pi^*\OO_X(mM)\otimes\mathcal I(h^{\otimes m})\big) = 0.$$
Theorem \ref{thm:DPS} implies that for all $p\geq 0$ and for all $m\gg0$:
$$H^p\big(Y,\OO_Y(K_Y+m\pi^*M)\otimes\mathcal I(h^{\otimes m})\big) = 0,$$
which together with \eqref{eq:metric1} and Serre duality yields 
\begin{equation}\label{eq:euler}
\textstyle\chi\big(Y,\OO_Y\big(\sum_{j=1}^r\lfloor m\lambda_j\rfloor D_j-m\pi^*M\big)\big) = 0
\end{equation}
for all $m\gg0$. Let $q$ be a positive integer such that $q\lambda_j\in\N$ for all $j$, and $D=\sum q\lambda_j D_j-q\pi^*M$. Then \eqref{eq:euler} implies
$$\chi\big(Y,\OO_Y(mD)\big) = 0\quad \text{for all }m\gg0.$$
But then $\chi(Y,\OO_Y) = 0$. Since $X$ has rational singularities, this implies $\chi(X,\OO_X) = 0$, a contradiction which finishes the proof.
\end{proof} 

When the metrics have all Lelong numbers zero, then the conclusion is even stronger:

\begin{thm}\label{thm:semipositive} 
Assume the existence of good models for klt pairs in dimensions at most $n-1$. 

Let $(X,\Delta)$ be a $\Q$-factorial projective terminal  pair of dimension $n$ such that $K_X$ is pseudoeffective and $K_X+\Delta$ is nef, and let $L$ be a nef $\Q$-Cartier divisor on $X$. 
\begin{enumerate}
\item[(i)] Assume that $K_X+\Delta+L$ has a singular metric with semipositive curvature current and vanishing Lelong numbers. If there exists a positive integer $m$ such that $m(K_X+\Delta+L)$ is Cartier and
$$\chi\big(X,\OO_X\big(m(K_X+\Delta+L)\big)\big)\neq0,$$
then for every $\Q$-divisor $L'$ with $L\equiv L'$ we have $\kappa(X,K_X+\Delta+L')\geq0$. 
\item[(ii)] If $K_X+\Delta$ and $K_X+\Delta+L$ have singular metrics with semipositive curvature currents and vanishing Lelong numbers (in particular, if $K_X+\Delta$ and $K_X+\Delta+L$ are semiample), and if $\chi(X,\OO_X)\neq0$, then for every $\Q$-divisor $L'$ with $L\equiv L'$ and for every rational number $t\geq0$, the divisor $K_X+\Delta+tL'$ is semiample.
\item[(iii)] If $K_X+\Delta$ has a singular metric $h$ with semipositive curvature current and vanishing Lelong numbers and if there exists a positive integer $\ell$ such that $\ell(K_X+\Delta)$ is Cartier and
$$\chi\big(X,\OO_X\big(\ell(K_X+\Delta)\big)\big)\neq0,$$
then every $\Q$-divisor $G$ with $K_X+\Delta\equiv G$ is semiample. 
\end{enumerate}
\end{thm}

\begin{proof}
In (i) and (ii), the assumptions imply that $K_X+\Delta+L'$ has a singular metric vanishing Lelong numbers and semipositive curvature current. Therefore, we may assume that $L=L'$. Denote $M:=K_X+\Delta+L$.

We first show (i). We follow closely the proof of Theorem \ref{thm:nv0}. We may assume that $X$ is $\Q$-factorial. 
 
Arguing by contradiction, assume that $\kappa(X,M) = - \infty$. By assumption, there exist a positive integer $\ell$ and a desingularisation $\pi\colon Y \to X$ such that $\ell M$ is Cartier and $\pi^*\OO_X(\ell M)$ has a singular metric $h$ with semipositive curvature current and vanishing Lelong numbers. This implies that 
$$ \mathcal I(h^{\otimes m}) = \OO_Y \quad\text{for all positive integers $m$.}$$
Then by Theorem \ref{thm:nonvanishingForms}, for all $p\geq 0$ and for all $m\gg0$ we have
$$ H^0\big(Y,\Omega^p_Y \otimes \pi^*\OO_X(m\ell M)\big)=0,$$
which together with Theorem \ref{thm:DPS} and Serre duality implies 
$$H^p\big(Y,\pi^*\OO_X(-m\ell M)\big) = 0\quad\text{for all }p\geq0\text{ and }m\gg0,$$
hence $\chi\big(Y,\pi^*\OO_X(m\ell M)\big) = 0$ for all $m\gg0$, hence for all $m$. Since $X$ has rational singularities, we deduce
$$ \chi\big(X,\OO_X(m\ell M)\big) = 0\quad\text{for all $m$},$$
a contradiction.

\medskip

Next we show (ii). By (i) and since $\chi(X,\OO_X)\neq0$, we have 
$$\kappa(X,K_X+\Delta)\geq0\quad\text{and}\quad \kappa(X,M)\geq0.$$
Then there exists an effective $\Q$-divisor $G$ such that $M\sim_\Q G$. Fix a rational number $0<\varepsilon\ll1$ such that the pair $(X,\Delta+\varepsilon G)$ is klt, and note that $K_X+\Delta+\varepsilon G$ has a singular metric $h$ with semipositive curvature current and vanishing Lelong numbers, since both $K_X+\Delta$ and $G$ have such metrics. By \cite[Theorem 5.1]{GM17}, the divisors $K_X+\Delta$ and $K_X+\Delta+\varepsilon G$ are semiample. Let $\varphi_0\colon X\to S_0$ and $\varphi_\varepsilon\colon X\to S_\varepsilon$ be the Iitaka fibrations of $K_X+\Delta$ and $K_X+\Delta+\varepsilon G$, respectively.

Let $C$ be a curve contracted by $\varphi_\varepsilon$. Then $(K_X+\Delta+\varepsilon G)\cdot C=0$ implies $$(K_X+\Delta)\cdot C=G\cdot C=0$$ as both $K_X+\Delta$ and $G$ are nef, hence by the Rigidity lemma \cite[Lemma 1.15]{Deb01} there exists a morphism $\xi\colon S_\varepsilon\to S_0$ such that $\varphi_0=\xi\circ\varphi_\varepsilon$.
\[
\xymatrix{ 
X \ar[d]_{\varphi_0} \ar[r]^{\varphi_\varepsilon} & S_\varepsilon \ar[dl]^\xi \\
S_0 & 
}
\]
There exist ample $\Q$-divisors $A_\varepsilon$ on $S_\varepsilon$ and $A_0$ on $S_0$ such that 
$$K_X+\Delta+\varepsilon G\sim_\Q\varphi_\varepsilon^*A_\varepsilon\quad\text{and}\quad K_X+\Delta\sim_\Q\varphi_0^*A_0,$$
and denote $L_\varepsilon=\frac{1}{\varepsilon}\big(A_\varepsilon-(1+\varepsilon)\xi^*A_0\big)$. Then
$$L\sim_\Q\varphi_\varepsilon^*L_\varepsilon,$$
and hence $L_\varepsilon$ is nef. Since for every $t\geq1$ we have
$$\textstyle K_X+\Delta+tL\sim_\Q \frac{1}{1+\varepsilon}\varphi_\varepsilon^*\big(A_\varepsilon+\big(t+(t-1)\varepsilon\big)L_\varepsilon\big)$$
and $A_\varepsilon+\big(t+(t-1)\varepsilon\big)L_\varepsilon$ is ample on $S_\varepsilon$, the divisor $K_X+\Delta+tL$ is semiample for every $t\geq1$. As $K_X+\Delta$ is semiample, it follows that $K_X+\Delta+tL$ is semiample for every $t\geq0$.

Finally, for (iii), there exists a $\Q$-divisor $L\equiv0$ such that $G=K_X+\Delta+L$. By (i) we have $\kappa(X,K_X+\Delta)\geq0$ and $\kappa(X,K_X+\Delta+L)\geq0$, and then the proof is the same as that of (ii).
\end{proof}

Theorem \ref{thm:nv0}(ii) and Theorem \ref{thm:semipositive}(iii) show that the Generalised Nonvanishing Conjecture, and in particular the techniques of this paper, have unexpected consequences even for the usual Nonvanishing Conjecture. We will partly generalise Theorem \ref{thm:semipositive} in Theorem \ref{thm:semipositive1} below.

In small dimensions, Theorem \ref{thm:semipositive} holds almost unconditionally by our results in \cite{LP18a}. We will present particularly interesting cases in Section \ref{sec:algsing2}.

\section{Metrics with generalised algebraic singularities, II}\label{sec:algsing2}

In this section we generalise the results from the previous section without assuming that we deal with minimal varieties. 

We use the same strategy as in Section \ref{sec:reductions}. The only issue is that we have to check that the reduction results preserve the assumptions on the singularities and the Euler-Poincar\'e characteristic. We thus have to analyse in detail the results of Section \ref{sec:reductions}.

We start with the following lemma.

\begin{lem}\label{lem:algsing}
Assume the termination of flips in dimensions at most $n-1$, and the Abundance Conjecture for klt pairs in dimensions at most $n-1$. 

Let $(X,\Delta)$ be a projective klt pair of dimension $n$ such that $K_X+\Delta$ is pseudoeffective. Suppose that $K_X+\Delta$ has a singular metric with generalised algebraic singularities and semipositive curvature current. If $\chi(X,\OO_X)\neq0$, then $\kappa(X,K_X+\Delta)\geq0$ and every sequence of flips of $(X,\Delta)$ terminates. 
\end{lem}

\begin{proof}
By \cite[Corollary 1.2]{HMX14} it suffices to show that $\kappa(X,K_X+\Delta)\geq0$. 

By Remark \ref{rem:WZD}, the pair $(X,\Delta)$ has a weak Zariski decomposition, hence it has a minimal model by \cite[Theorem 1.5]{Bir12b}; note that we only need the termination of klt flips in dimensions at most $n-1$, since the termination of klt flips implies the termination of dlt flips by the special termination \cite{Fuj07a}. Let $(X_m,\Delta_m)$ be a minimal model of $(X,\Delta)$. Then $K_{X_m}+\Delta_m$ has a singular metric with generalised algebraic singularities and semipositive curvature current by Corollary \ref{cor:algsinMMP}. Since klt pairs have rational singularities, the Euler-Poincar\'e characteristic of the structure sheaf is preserved during any MMP of klt pairs, hence 
$$\chi(X_m,\OO_{X_m})=\chi(X,\OO_X)\neq0.$$ By passing to a terminalisation, we may assume that $(X_m,\Delta_m)$ is a terminal pair. If $X$ is uniruled, then $\kappa(X_m,K_{X_m}+\Delta_m)\geq0$ by \cite[Theorem 3.3]{DL15}. If $X$ is not uniruled, then $K_X$ is pseudoeffective by \cite[Corollary 0.3]{BDPP}, hence $\kappa(X,K_X+\Delta)=\kappa(X_m,K_{X_m}+\Delta_m)\geq0$ by Theorem \ref{thm:nv0}.
\end{proof}

\begin{thm}\label{thm:Supportalgsings}
Assume the termination of flips in dimensions at most $n-1$, and the Abundance Conjecture for klt pairs in dimensions at most $n-1$. 

Let $(X,\Delta)$ be an $n$-dimensional $\Q$-factorial terminal pair such that $K_X+\Delta$ is pseudoeffective, and let $L$ be a nef $\Q$-divisor on $X$. Suppose that $K_X+\Delta+L$ has a singular metric with generalised algebraic singularities and semipositive curvature current and that $\chi(X,\OO_X)\neq0$. Assume that there exists an effective $\Q$-divisor $D$ such that
$$K_X + \Delta \sim_\Q D \geq 0 \quad \text{and} \quad \Supp\Delta\subseteq\Supp D.$$
Then for every $\Q$-divisor $L_0$ with $L\equiv L_0$ there exists a rational $t_0>0$ such that we have $\kappa(X,K_X+\Delta+tL_0)\geq0$ for every $0<t\leq t_0$.
\end{thm}

\begin{proof}
The assumptions imply that $K_X+\Delta+L_0$ has a singular metric with generalised algebraic singularities and semipositive curvature current. Therefore, we may assume that $L=L_0$. We follow closely the proof of Theorem \ref{thm:Support}. 

\medskip

\emph{Step 1.}
Let $f\colon X'\to X$ be a log resolution of the pair $(X,D)$, and we write
$$K_{X'}+\Delta_1\sim_\Q f^*(K_X+\Delta)+F',$$
where $\Delta_1$ and $F'$ are effective $\Q$-divisors with no common components. Denote $D_1:=f^*D+F'$ and $L':=f^*L$, so that $K_{X'}+\Delta_1\sim_\Q D_1$, and the divisors 
$$K_{X'}+\Delta_1\quad\text{and}\quad K_{X'}+\Delta_1+L'$$
have singular metrics with generalised algebraic singularities and semipositive curvature currents. We may also assume that these metrics descend to $X'$, see Section \ref{sec:prelim}.

Then Step 1 of the proof of Theorem \ref{thm:Support} shows that there exists a $\Q$-divisor $H\geq0$ with $\Supp H=\Supp D_1$ such that, if we denote $\Delta':=\Delta_1+H$ and $D':=D_1+H$, then the pair $(X',\Delta')$ is log smooth with $\Delta'$ reduced, and
$$K_{X'}+\Delta'\sim_\Q D'\geq0\quad \text{with}\quad \Supp\Delta'=\Supp D'.$$
In particular, there exists $0<\eta_0\ll1$ such that for all $0\leq\eta\leq\eta_0$ the divisors 
$$K_{X'}+(1-\eta)\Delta'\quad \text{and}\quad K_{X'}+(1-\eta)\Delta'+L'$$
have singular metrics with generalised algebraic singularities and semipositive curvature currents which descend to $X'$. As in Step 1 of the proof of Theorem \ref{thm:Support}, it suffices to show that there exists a rational $t_0'$ such that
\begin{equation}\label{eq:claimclaim1}
\kappa(X',D'+tL')\geq0\quad\text{for every }0<t\leq t_0'.
\end{equation}

\medskip

\emph{Step 2.}
As in Step 2 of the proof of Theorem \ref{thm:Support} we construct a cyclic covering $\pi\colon X''\to X'$, a reduced divisor $\Delta_1''$ on $X''$ such that 
$$K_{X''}+\Delta_1''=\pi^*(K_{X'}+\Delta'),$$
and an effective Cartier divisor $D_1'':=\pi^*D'$ on $X''$ such that
$$K_{X''}+\Delta_1''\sim D_1''\geq0\quad \text{with}\quad \Supp\Delta_1''=\Supp D_1''.$$
Pick a rational number $0<\delta\leq\eta_0$ such that 
$$\Delta'':=\Delta_1''-\delta\pi^*\Delta'\geq0,\ D'':=D_1''-\delta\pi^*\Delta'\geq0\text{ and } \Supp D''=\Supp D_1''.$$
Then
\begin{equation}\label{eq:331}
K_{X''}+\Delta''=\pi^*\big(K_{X'}+(1-\delta)\Delta'\big)\quad\text{and}\quad K_{X''}+\Delta''\sim_\Q D'',
\end{equation}
and therefore, the pair $(X'',\Delta'')$ is klt by \cite[Proposition 5.20]{KM98}. Denote $L'':=\pi^*L'$. Then by Step 1 the divisors 
$$K_{X''}+\Delta''\quad\text{and}\quad K_{X''}+\Delta''+L''$$
have singular metrics with generalised algebraic singularities and semipositive curvature currents. Moreover, $X''$ is not uniruled and $\kappa(X'',K_{X''})\geq0$ as in Step 4 of the proof of Theorem \ref{thm:Support}.

\medskip

\emph{Step 3.}
Since we are assuming the termination of flips in dimension $n-1$, \cite[Corollary 1.2]{HMX14} implies that any sequence of flips of the pair $(X'',\Delta'')$ terminates. Therefore, by Proposition \ref{pro:contocon} there exists an $L''$-trivial $(K_{X''}+\Delta'')$-MMP $\varphi\colon X''\dashrightarrow Y$ such that $K_Y+\varphi_*\Delta''+m\varphi_*L''$ is nef for some fixed $m\gg0$. Denote $\Delta_Y:=\varphi_*\Delta''$ and $L_Y:=\varphi_*L''$. 

As in the previous paragraph, any sequence of flips of the pair $(Y,\Delta_Y)$ terminates. As in Step 1 and the beginning of Step 2 of the proof of \cite[Lemma 5.2]{LP18a}, there is a birational contraction
$$\theta\colon (Y,\Delta_Y)\dashrightarrow (Y_{\min},\Delta_{\min})$$ 
which is a composition of a sequence of operations of a $(K_Y+\Delta_Y)$-MMP, a positive rational number $\lambda$ and a divisor $L_{\min}:=\theta_*(mL_Y)$ on $Y_{\min}$ such that:
\begin{enumerate}
\item[(a)] $K_{Y_{\min}}+\Delta_{\min}$ is nef and $s(K_{Y_{\min}}+\Delta_{\min})+L_{\min}$ is nef for all $s\geq\lambda$,
\item[(b)] the map $\theta$ is $\big(s(K_Y+\Delta_Y)+mL_Y)$-negative for $s>\lambda$.
\end{enumerate}

Set $\mu:=\max\{\lambda,m\}$. Then by Step 1 the divisor 
$$M:=(\mu+1)(K_{X'}+(1-\delta)\Delta')+mL'$$
has a singular metric with generalised algebraic singularities and semipositive curvature current which descends to $X'$. 

\medskip

\emph{Step 4.}
By assumption, we have the Siu decomposition
$$M \equiv \Theta + \sum \lambda_i [D_i],$$ 
where $\Theta\geq0$ has vanishing Lelong numbers, $\lambda_i$ are positive rational numbers and $\sum D_i$ is a simple normal crossings divisor. 

Let $(p,q)\colon T\to X''\times Y_{\min}$ be a resolution of indeterminacies of the birational map $\theta\circ\varphi\colon X''\dashrightarrow Y_{\min}$. Set $g := \pi\circ p\colon T\to X'$. 
\[
\xymatrix{ 
T \ar[r]_p \ar@/^1.0pc/[rrr]^q \ar[dr]_g & X'' \ar[d]^{\pi} \ar@{-->}[rr]_{\theta\circ\varphi} && Y_{\min} \\
 & X' &&
}
\]
Then 
$$ M_T := g^*M \equiv g^*\Theta + \sum \lambda_i g^*[D_i].$$
We may assume that $\sum \lambda_i g^*D_i$ has simple normal crossings support. 

\medskip

\emph{Step 5.}
Denote 
$$M'':=\pi^*M\quad\text{and}\quad M_{\min}:=(\theta\circ\varphi)_*M''=q_*M_T.$$
Then by \eqref{eq:331} we have
$$M''=(\mu+1)(K_{X''}+\Delta'')+mL'',$$
and hence
$$M_{\min}=(\mu+1)(K_{Y_{\min}}+\Delta_{\min})+L_{\min}.$$
The divisor $M_{\min}$ is nef by (a).

\medskip

We claim that $\kappa(Y_{\min},M_{\min})\geq0$. The claim immediately implies the theorem: indeed, then $\kappa(X',M')=\kappa(X'',M'')\geq0$ by (b), which gives \eqref{eq:claimclaim1}, as desired.

\medskip

\emph{Step 6.}
It remains to prove the claim. Assume by contradiction, that $\kappa(Y_{\min},M_{\min})=-\infty$. Since $M_T=p^*M''$, by (b) there exists an effective $q$-exceptional $\Q$-divisor $E$ on $T$ such that
$$M_T\sim_\Q q^*M_{\min}+E.$$
The proof of Lemma \ref{lem:Siu} shows that $\sum \lambda_i g^*D_i\geq E$, so that
$$\textstyle q^*M_{\min}\equiv g^*\Theta + \big(\sum \lambda_i g^*D_i-E\big).$$
Set 
$$\textstyle N := (\sum \lambda_i g^*D_i-E) - q^*M_{\min}.$$ 
Then as in the proof of Theorem \ref{thm:nv0}, for all $p \geq 0$ and all $\ell$ sufficiently divisible we have
\begin{equation}\label{eq:N}
H^p\big(T, \OO_T(\ell N)\big) = 0.
\end{equation}
Since 
$$\textstyle N \sim_\Q \sum \lambda_i g^*D_i-M_T=p^*\pi^*\big(\sum \lambda_i D_i-M\big)$$
and since $X''$ has rational singularities by Step 2, Equation \eqref{eq:N} implies
$$\textstyle H^p\big(X'', \OO_{X''}\big(\ell\pi^*\big(\sum \lambda_iD_i - M\big)\big)\big) = 0$$ 
for all $p \geq 0$ and all $\ell$ sufficiently divisible. As $\pi$ is finite, \cite[Lemma 4.1.14]{Laz04} gives
$$\textstyle H^p\big(X', \OO_{X'}\big(\ell\big(\sum \lambda_i D_i - M\big)\big)\big) = 0$$
for all $p \geq 0$ and all $\ell$ sufficiently divisible. As in the proof of Theorem \ref{thm:nv0}, this implies 
$$\chi(X',\OO_{X'})=0,$$
hence $\chi(X,\OO_X)=0$ since $X$ has rational singularities. This is a contradiction which proves the claim and finishes the proof of the theorem.
\end{proof}

\begin{thm}\label{thm:tau<1algsings}
Assume the termination of flips in dimensions at most $n-1$, and the Abundance Conjecture in dimensions at most $n-1$.

Let $(X,\Delta)$ be an $n$-dimensional $\Q$-factorial terminal pair such that $K_X+\Delta$ is pseudoeffective, let $L$ be a nef $\Q$-divisor on $X$, and assume that there exists a minimal model of $(X,\Delta)$. Suppose that $K_X+\Delta+L$ has a singular metric with generalised algebraic singularities and semipositive curvature current and that $\chi(X,\OO_X)\neq0$. Assume that $K_X$ is not pseudoeffective.

Then for every $\Q$-divisor $L_0$ with $L\equiv L_0$ there exists a rational $t_0>0$ such that we have $\kappa(X,K_X+\Delta+tL_0)\geq0$ for every $0<t\leq t_0$.
\end{thm}

\begin{proof}
The assumptions imply that $K_X+\Delta+L_0$ has a singular metric with generalised algebraic singularities and semipositive curvature current. Therefore, we may assume that $L=L_0$. The proof is analogous to that of Lemma \ref{lem:Knotpsef}, by invoking Theorem \ref{thm:Supportalgsings} instead of Theorem \ref{thm:Support}, and by invoking Lemma \ref{lem:algsing} instead of Lemma \ref{lem:termination}.
\end{proof}

Finally, we have all the ingredients to prove Theorem \ref{thm:C}; Corollary \ref{cor:D} is an immediate consequence.

\begin{proof}[Proof of Theorem \ref{thm:C}]
The assumptions imply that $K_X+\Delta+L'$ has a singular metric with generalised algebraic singularities and semipositive curvature current. Therefore, we may assume that $L=L'$. 

\medskip

\emph{Step 1.}
Note that the assumptions imply that for each $0\leq t\leq1$, the divisor $K_X+\Delta+tL$ has a singular metric with generalised algebraic singularities and semipositive curvature current.

By Lemma \ref{lem:algsing} we have $\kappa(X,K_X+\Delta)\geq0$ and every sequence of flips of $(X,\Delta)$ terminates. By Proposition \ref{pro:contocon}, by Lemma \ref{lem:MMPnumeff} and by Corollary \ref{cor:algsinMMP} we may assume that $K_X+\Delta+L$ is nef.

We follow the proof of Theorem \ref{thm:induction2} closely. However, note that the conclusion in part (i) is stronger than that of Theorem \ref{thm:mainreduction}(i). We are thus somewhat more careful how we conduct the proof.

\medskip

\emph{Step 2.}
It suffices to show (i). Indeed, if $n(X,K_X+\Delta+L)<n$, then in this case of part (ii) we conclude by Theorem \ref{thm:B}. Thus, we may (for part (ii)) assume that $n(X,K_X+\Delta+L)=n$. Then this is done analogously as in Step 2 of the proof of Theorem \ref{thm:induction2}.

\medskip

\emph{Step 3.}
Therefore, from now on we prove (i). Then analogously as in Step 3 of the proof of Theorem \ref{thm:induction2}, by using that metrics with generalised algebraic singularities and semipositive curvature currents are preserved by an MMP by Corollary \ref{cor:algsinMMP}, we show that we may additionally assume the following:

\medskip

\emph{Assumption 1.} The pair $(X,\Delta)$ is $\Q$-factorial and terminal, and $K_X+\Delta$ is nef.

\medskip

We only note that we do not prove (but we also do not need) item (c) in Step 3 of the proof of Theorem \ref{thm:induction2}.

Now, if $K_X$ is pseudoeffective, then (i) follows from Theorem \ref{thm:nv0}. If $K_X$ is not pseudoeffective, then (i) follows from Theorem \ref{thm:tau<1algsings}. This completes the proof.
\end{proof}

When the variety is not uniruled, or if we add an additional assumption, we get more:

\begin{thm}
Assume the termination of flips in dimensions at most $n-1$, and the Abundance Conjecture in dimensions at most $n-1$.

Let $(X,\Delta)$ be a projective klt pair of dimension $n$ such that $K_X+\Delta$ is pseudoeffective, and let $L$ be a nef $\Q$-divisor on $X$. Suppose that $K_X+\Delta$ and $K_X+\Delta+L$ have singular metrics with generalised algebraic singularities and semipositive curvature currents, and that $\chi(X,\OO_X)\neq0$.

Assume either that $X$ is not uniruled, or assume the semiampleness part of the Abundance Conjecture in dimension $n$. Then there exists a positive rational number $t_0$ such that for every $\Q$-divisor $L'$ with $L\equiv L'$ we have
$$\kappa(X,K_X+\Delta+tL')\geq0\quad\text{for every }0\leq t\leq t_0.$$
\end{thm}

\begin{proof}
When $X$ is not uniruled, the proof is the same as the proof of Theorem \ref{thm:C}, by noticing that we do not need Theorem \ref{thm:tau<1algsings}.

\medskip

Now, for the rest of the proof we do not assume that $X$ is not uniruled, but we assume the semiampleness part of the Abundance Conjecture in dimension $n$. Fix a $\Q$-divisor $L'$ with $L\equiv L'$. We will show that there exists a rational number $t_0>0$, independent of $L'$, such that $\kappa(X,K_X+\Delta+tL')\geq0$ for all $0\leq t\leq t_0$. 

By Theorem \ref{thm:C} there exists a positive rational number $r_0'$ such that 
$$\kappa(X,K_X+\Delta+tL')\geq0\quad\text{for all }0\leq t\leq r_0'.$$ 

Since we are assuming the termination of flips in dimension $n-1$, by \cite[Corollary 1.2]{HMX14} every sequence of flips of $(X,\Delta)$ terminates. Therefore, by Proposition \ref{pro:contocon} there exists an $L$-trivial (hence $L'$-trivial) $(K_X+\Delta)$-MMP $\varphi\colon X\dashrightarrow Y$ such that $K_Y+\varphi_*\Delta+m\varphi_*L$ is nef for some fixed $m\gg0$. Denote $\Delta_Y:=\varphi_*\Delta$, $L_Y:=\varphi_*L$ and $L_Y':=\varphi_*L'$. 

As in the previous paragraph, any sequence of flips of the pair $(Y,\Delta_Y)$ terminates. As in Step 1 and the beginning of Step 2 of the proof of \cite[Lemma 5.2]{LP18a}, there is a birational contraction
$$\theta\colon (Y,\Delta_Y)\dashrightarrow (Y_{\min},\Delta_{\min})$$ 
which is a composition of a sequence of operations of a $(K_Y+\Delta_Y)$-MMP, a positive rational number $\lambda$ and a divisor $L_{\min}:=\theta_*L_Y$ on $Y_{\min}$ such that:
\begin{enumerate}
\item[(a)] $K_{Y_{\min}}+\Delta_{\min}$ is nef and $s(K_{Y_{\min}}+\Delta_{\min})+mL_{\min}$ is nef for all $s\geq\lambda$,
\item[(b)] the map $\theta$ is $\big(s(K_Y+\Delta_Y)+mL_Y)$-negative for $s>\lambda$.
\end{enumerate}
Denote $L_{\min}':=\theta_*L_Y'$. Then clearly:
\begin{enumerate}
\item[(a$_0$)] $s(K_{Y_{\min}}+\Delta_{\min})+mL_{\min}'$ is nef for all $s\geq\lambda$,
\item[(b$_0$)] the map $\theta$ is $\big(s(K_Y+\Delta_Y)+mL_Y')$-negative for $s>\lambda$.
\end{enumerate}

Denote $M:=\lambda(K_{Y_{\min}}+\Delta_{\min})+mL_{\min}'$. Then $M$ is nef by (a$_0$). Since $\kappa(Y_{\min},K_{Y_{\min}}+\Delta_{\min}+t L_{\min}')\geq0$ for $0\leq t\leq r_0'$, there exist a rational number $0<\delta\leq \min\{r_0',m/\lambda\}$ and an effective divisor $G'$ such that
$$K_{Y_{\min}}+\Delta_{\min}+\delta L_{\min}'\sim_\Q G'.$$
Then $G'$ is nef by (a$_0$), and set $G:=\frac{m}{m-\delta\lambda}G'$. We have
$$\textstyle K_{Y_{\min}}+\Delta_{\min}+\frac{\delta}{m-\delta\lambda}M=\frac{m}{m-\delta\lambda}(K_{Y_{\min}}+\Delta_{\min}+\delta L_{\min}')\sim_\Q G.$$
Fix a rational number $0<\varepsilon\ll1$ such that the pair $(Y_{\min},\Delta_{\min}+\varepsilon G)$ is klt, and note that $K_{Y_{\min}}+\Delta_{\min}+\varepsilon G$ is nef since $K_{Y_{\min}}+\Delta_{\min}$ and $G$ are nef. Since $\kappa(Y_{\min},K_{Y_{\min}}+\Delta_{\min}+\varepsilon G)\geq0$ and since we are assuming semiampleness part of the Abundance Conjecture in dimension $n$, the divisors $K_{Y_{\min}}+\Delta_{\min}$ and $K_{Y_{\min}}+\Delta_{\min}+\varepsilon G$ are semiample. Then we conclude that $K_{Y_{\min}}+\Delta_{\min}+tM$ is semiample for all $t\geq0$ as in the proof of Theorem \ref{thm:semipositive}(ii). In particular, the divisor $(\lambda+1)(K_{Y_{\min}}+\Delta_{\min})+mL_{\min}'$ is semiample. But by (b$_0$) this implies 
$$\textstyle \kappa\big(X,K_X+\Delta+\frac{m}{\lambda+1}L'\big)=\kappa\big(Y,K_Y+\Delta_Y+\frac{m}{\lambda+1}L_Y'\big)\geq0.$$
This finishes the proof.
\end{proof}

We deduce the following surprising fact from Theorem \ref{thm:C} when $L=0$.

\begin{cor}\label{cor:numequiv} 
Assume the termination of flips in dimensions at most $n-1$, the Abundance Conjecture in dimensions at most $n-1$, and the semiampleness part of the Abundance Conjecture in dimension $n$.

Let $(X,\Delta)$ be a projective klt pair of dimension $n$ such that $K_X+\Delta$ is pseudoeffective. Suppose that $K_X+\Delta$ has a singular metric with generalised algebraic singularities and semipositive curvature current. If $\chi(X,\OO_X)\neq0$, then for every $\Q$-Cartier divisor $L$ with $L\equiv 0$ and for every rational number $t\geq0$ we have $\kappa(X,K_X+\Delta+tL)\geq0$. 
\end{cor}

\begin{proof}
Fix a $\Q$-divisor $L$ with $L\equiv 0$. Then by Theorem \ref{thm:C} there exists a positive rational number $t_0$ such that $\kappa(X,K_X+\Delta+tL)\geq0$ for all $0\leq t\leq t_0$. Since we are assuming the termination of flips in dimension $n-1$, by \cite[Corollary 1.2]{HMX14} every sequence of flips of $(X,\Delta)$ terminates. 

Let $\varphi\colon X\dashrightarrow Y$ be any $(K_X+\Delta)$-MMP which terminates with a minimal model $(Y,\Delta_Y)$, where $\Delta_Y:=\varphi_*\Delta$. Denote $L_Y:=\varphi_*L$, and note that this MMP is $L$-trivial. Then there 
exists an effective divisor $G$ such that $K_Y+\Delta_Y+t_0L_Y\sim_\Q G$. Fix a rational number $0<\varepsilon\ll1$ such that the pair $(Y,\Delta_Y+\varepsilon G)$ is klt, and note that $K_Y+\Delta_Y+\varepsilon G$ is nef since both $K_Y+\Delta_Y$ and $G$ are nef. Since $\kappa(Y,K_Y+\Delta_Y+\varepsilon G)\geq0$ and since we are assuming semiampleness part of the Abundance Conjecture in dimension $n$, the divisors $K_Y+\Delta_Y$ and $K_Y+\Delta_Y+\varepsilon G$ are semiample. Then we conclude that $K_Y+\Delta_Y+tL_Y$ is semiample for all $t\geq0$ as in the proof of Theorem \ref{thm:semipositive}(ii), which implies the result.
\end{proof}

We also have the following partial generalisation of Theorem \ref{thm:semipositive}, where we do not assume that the variety is not uniruled:

\begin{thm}\label{thm:semipositive1} 
Assume the termination of flips in dimensions at most $n-1$, and the Abundance Conjecture in dimensions at most $n-1$. 

Let $(X,\Delta)$ be a $\Q$-factorial projective klt pair of dimension $n$, and let $L$ be a nef $\Q$-Cartier divisor on $X$. Assume that $K_X+\Delta$ and $K_X+\Delta+L$ have singular metrics with semipositive curvature currents whose Lelong numbers are all zero. If $\chi(X,\OO_X)\neq0$, then for every $\Q$-divisor $L'$ with $L\equiv L'$ and for every rational number $t\geq0$, the divisor $K_X+\Delta+tL'$ is semiample.
\end{thm}

\begin{proof}
Fix a $\Q$-divisor $L'$ with $L\equiv L'$. Then by Theorem \ref{thm:C} there exists a positive rational number $t_0$ such that $\kappa(X,K_X+\Delta+tL')\geq0$ for all $0\leq t\leq t_0$. We may assume that $t_0\leq1$. Then there exists an effective divisor $G$ such that $K_X+\Delta+t_0L\sim_\Q G$. Fix a rational number $0<\varepsilon\ll1$ such that the pair $(X,\Delta+\varepsilon G)$ is klt, and note that $K_Y+\Delta_Y+\varepsilon G$ has a singular metric with semipositive curvature current whose Lelong numbers are all zero since both $K_X+\Delta$ and $G$ do. Since $\kappa(X,K_X+\Delta+\varepsilon G)\geq0$, by \cite[Theorem 5.1]{GM17} the divisors $K_X+\Delta$ and $K_X+\Delta+\varepsilon G$ are semiample. Then we conclude that $K_X+\Delta+tL'$ is semiample for all $t\geq0$ as in the proof of Theorem \ref{thm:semipositive}(ii).
\end{proof}

We summarise the situation for surfaces and threefolds in the following two results.

\begin{cor}\label{cor:surfaces}
Let $(X,\Delta)$ be a projective klt surface pair such that $K_X+\Delta$ is pseudoeffective, and let $L$ be a nef $\Q$-Cartier divisor on $X$. Assume that $\chi(X,\OO_X)\neq0$.
\begin{enumerate}
\item[(i)] Then for every $\Q$-divisor $D$ with $K_X+\Delta\equiv D$ we have $\kappa(X,D)\geq0$.
\item[(ii)] If $K_X+\Delta+L$ is nef, then every $\Q$-divisor $D$ with $K_X+\Delta+L\equiv D$ is semiample. 
\end{enumerate}
\end{cor}

\begin{proof}
Since $\dim X=2$, the assumptions of Corollary \ref{cor:numequiv} are satisfied, which gives (i). Furthermore, the assumptions of Theorem \ref{thm:semipositive1} are satisfied by \cite[Corollary C]{LP18a}, which shows (ii).
\end{proof}

\begin{cor}\label{cor:threefolds}
Let $(X,\Delta)$ be a projective klt threefold pair such that $K_X+\Delta$ is pseudoeffective, and let $L$ be a nef $\Q$-Cartier divisor on $X$. Assume that $\chi(X,\OO_X)\neq0$.
\begin{enumerate}
\item[(i)] Then for every $\Q$-divisor $D$ with $K_X+\Delta\equiv D$ we have $\kappa(X,D)\geq0$.
\item[(ii)]  If $K_X+\Delta$ is nef, then every $\Q$-divisor $D$ with $K_X+\Delta\equiv D$ is semiample. 
\item[(iii)] If $\nu(X,K_X+\Delta)>0$ and $K_X+\Delta$ is nef, then every $\Q$-divisor $D$ with $K_X+\Delta+L\equiv D$ is semiample. 
\end{enumerate}
\end{cor}

\begin{proof}
Since $\dim X=3$, the assumptions of Corollary \ref{cor:numequiv} are satisfied, which gives (i). For (ii), the divisor $K_X + \Delta$ is semiample by the Abundance Conjecture on threefolds, hence Theorem \ref{thm:semipositive1} applies. Finally, for (iii) the assumptions of Theorem \ref{thm:semipositive1} are satisfied by \cite[Corollary D]{LP18a}. 
\end{proof}

\section{Numerical dimension 1} \label{sec:num1}

In case of numerical dimension one, the results of Section 5 hold unconditionally. More precisely, the following theorem and Theorem \ref{thm:G} generalise \cite[Theorem 6.3 and Theorem 6.7]{LP18}.

\begin{thm} \label{thm:nu=1}
Let $(X,\Delta)$ be a projective terminal pair of dimension $n$ such that $K_X + \Delta$ is nef. Let $L$ be a nef Cartier divisor on $X$ and let $t$ be a positive integer such that $t(K_X+\Delta)$ is Cartier. Denote $M=t(K_X+\Delta)+L$. Assume that $\nu(X,M) = 1$ and let $\pi\colon Y\to X$ be a resolution of $X$. Assume that for some positive integer $p$ we have  
$$ H^0\big(Y,(\Omega^1_Y)^{\otimes p} \otimes \OO_Y(m\pi^*M)\big) \neq 0$$
for infinitely many integers $m$. Then $M$ is num-effective. 
\end{thm} 

\begin{proof}
By repeating verbatim the proof of Theorem \ref{thm:nonvanishingForms}, we obtain that there exist a resolution $\pi\colon Y\to X$, a positive integer $\ell$, infinitely many positive integers $m$, effective Weil divisors $N_m$ on $Y$ and a pseudoeffective divisor $F$ on $Y$ such that
\begin{equation}\label{eq:repeat}
\pi_*N_m+(\pi_*F+\ell\Delta)\sim_\Q mM +\ell (K_X+\Delta).
\end{equation}
Assume that there exist positive integers $m_1<m_2$ such that $N_{m_1}=N_{m_2}=0$. Then \eqref{eq:repeat} gives
$$(m_2-m_1)M\sim_\Q0,$$
and the result follows. Therefore, there exists a positive integer $m_0$ such that $N_{m_0}\neq0$. By adding $\frac{\ell}{t}L$ to the relation \eqref{eq:repeat} for $m=m_0$, we obtain
$$\textstyle \pi_*N_{m_0}+\big(\pi_*F+\ell\Delta+\frac{\ell}{t}L\big)\sim_\Q \big(m_0+\frac{\ell}{t}\big) M.$$
We then conclude by \cite[Theorem 6.1]{LP18}.
\end{proof}

\begin{proof}[Proof of Theorem \ref{thm:G}]
Fix a rational number $t\geq0$ and set $M:=K_X+\Delta+tL$. We will show that $M$ is num-effective.

Assume first that $t>0$. Then we first claim that $\nu(X,M)=1$. Indeed, if $t\leq1$, then 
$$1=\nu\big(X,t(K_X+\Delta+L)\big)\leq\nu(X,M)\leq\nu(X,K_X+\Delta+L)=1$$
as $(1-t)(K_X+\Delta)$ and $(1-t)L$ is pseudoeffective, and similarly if $t\geq1$.

If there exist a resolution $\pi\colon Y\to X$, a positive integer $m$ such that $mM$ is Cartier, and a singular metric $h$ on $\pi^*\OO_X(mM)$ with semipositive curvature current, such that $\mathcal I(h)\neq\OO_Y$, then the result follows from \cite[Theorem 6.5]{LP18}. 

Otherwise, pick a resolution $\pi\colon Y\to X$, a positive integer $m$ such that $mM$ is Cartier, and a singular metric $h$ on $\pi^*\OO_X(mM)$ with semipositive curvature current. Then the proof is the same as that of Theorem \ref{thm:nv0}, by invoking Theorem \ref{thm:nu=1} instead of Theorem \ref{thm:nonvanishingForms}.

Finally, if $t=0$, then $\nu(X,M)\leq\nu(X,K_X+\Delta+L)=1$ since $L$ is pseudoeffective. If $\nu(X,M)=0$, then $M\equiv 0$. Otherwise, $\nu(X,M)=1$, and then $M$ is num-effective by the argument above, by setting $L=0$.
\end{proof}

\section{Maps to abelian varieties}\label{sec:abelian}

In this section we apply our methods to varieties admitting a morphism to an abelian variety. 

\begin{thm}\label{thm:overabelian}
Assume the Generalised Nonvanishing Conjecture in dimensions at most $n-1$, the Abundance Conjecture in dimensions at most $n-1$ and the termination of flips in dimensions at most $n-1$. 

Let $(X,\Delta)$ be a projective klt pair of dimension $n$ such that $K_X+\Delta$ is pseudoeffective, and let $L$ be a nef $\Q$-divisor on $X$. Assume that there exists a nonconstant morphism $\alpha\colon X\to A$ to an abelian variety $A$. Then $K_X+\Delta+tL$ is num-effective for all $t\geq0$.
\end{thm}

\begin{proof} 
By replacing $L$ by some multiple, we may assume that $L$ is Cartier. Consider the Stein factorisation of $\alpha$:
$$X \, {\buildrel {\beta} \over {\longrightarrow} }\, Y \, {\buildrel {\gamma} \over {\longrightarrow} } \, A.$$
Let $F$ be a general fibre of $\beta$. Then $\kappa\big(F,(K_X+\Delta)|_F\big)\geq0$ by the Generalised Nonvanishing Conjecture in dimension $\dim F$ and by \cite[Theorem 0.1]{CKP12}, hence $\kappa(X,K_X+\Delta)\geq0$ by \cite[Lemma 3.5]{Hu18}. Therefore, it suffices to prove that $K_X+\Delta+tL$ is num-effective for any $t>2n$. 

Fix any $m>2n$. By \cite[Corollary 1.2]{HMX14}, any sequence of flips of the pair $(X,\Delta)$ terminates. Therefore, by Proposition \ref{pro:contocon} there exists an $L$-trivial $(K_X+\Delta)$-MMP $\varphi\colon X\dashrightarrow Y$ such that $K_Y+\varphi_*\Delta+m\varphi_*L$ is nef. 

We claim that this MMP is, in fact, a relative MMP over $Y$ (hence over $A$). Indeed, it suffices to show this on the first step of the MMP, as the rest is analogous. Let $c_R\colon X\to Z$ be the contraction of a $(K_X+\Delta)$-negative extremal ray, and let $H$ be an ample divisor on $Y$. Then by the Cone theorem \cite[Theorem 3.7]{KM98}, $R$ is spanned by the class of some rational curve $C$ on $X$. Since abelian varieties contain no rational curves, neither does $Y$, and therefore $C$ has to be contracted by $\beta$. In particular, $\beta^*H\cdot C=0$. If $C'$ is any other curve contracted by $c_R$, then $C'$ is numerically proportional to $C$. Hence, $\beta^*H\cdot C'=0$ and so $C'$ is contracted by $\beta$. By the Rigidity lemma \cite[Lemma 1.15]{Deb01}, the morphism 
$\beta$ factors through $c_R$, which proves the claim.

By Lemma \ref{lem:MMPnumeff} it suffices to show that $K_Y+\varphi_*\Delta+m\varphi_*L$ is num-effective. Therefore, by replacing $X$ by $Y$, $\Delta$ by $\Delta_Y$ and $L$ by $mL_Y$, we may assume additionally that $K_X+\Delta+L$ is nef. We need then to show that $K_X+\Delta+L$ is num-effective.

If $n(X,K_X+\Delta+L)<n$, then $K_X+\Delta+L$ is num-effective by Theorem \ref{thm:B}. 

If $n(X,K_X+\Delta+L)=n$, then for a very general fibre $F$ of $\beta$ we have
$$n\big(F,(K_X+\Delta+L)|_F\big)=\dim F$$
by \cite[Lemma 2.10]{LP18a}. Then $(K_X+\Delta+L)|_F$ is big by \cite[Lemma 5.2]{LP18a}, and hence $K_X+\Delta+L$ is $\beta$-big. Then $K_X+\Delta+L$ is num-effective by \cite[Theorem 4.1]{BirChen15}.
\end{proof}

\bibliographystyle{amsalpha}

\bibliography{biblio}
\end{document}